\documentclass[11pt,a4paper]{article}

\usepackage{titlesec}
\usepackage{fancyhdr}
\usepackage{a4wide}
\usepackage{graphicx}
\usepackage{float}
\usepackage{amssymb}
\usepackage{amsmath}
\usepackage{amsfonts}
\usepackage{dsfont}
\usepackage{amsthm}
\usepackage{color}
\usepackage{mathrsfs}
\usepackage{array}
\usepackage{eucal}
\usepackage{tikz}
\usepackage[T1]{fontenc}
\usepackage{inputenc}
\usepackage[english]{babel}
\usepackage{lmodern}
\usepackage{hyperref}
\usepackage{mathtools}
\usepackage{geometry}
\usepackage{changepage}
\usepackage{bm}
\changepage{0pt}{}{}{}{}{0pt}{}{0pt}{6pt}
\usepackage[numbers]{natbib}
\usepackage{bbm}
\usepackage{pgfplots}

\hypersetup{
pdfpagemode=none,
pdftoolbar=true,        
pdfmenubar=true,        
pdffitwindow=false,     
pdfstartview={Fit},    
pdftitle={Eigenvalue falls in thin broken quantum strips},    
pdfauthor={L. Chesnel, S.A. Nazarov},     
pdfsubject={},  
pdfcreator={L. Chesnel, S.A. Nazarov},   
pdfproducer={}, 
pdfkeywords={}, 
pdfnewwindow=true,      
colorlinks=true,       
linkcolor=magenta,          
citecolor=red,        
filecolor=cyan,      
urlcolor=blue           
}

\newcommand{\dsp}{\displaystyle}

\newcommand{\eps}{\varepsilon}
\newcommand{\om}{\omega}
\newcommand{\Om}{\Omega}
\newcommand{\mrm}[1]{\mathrm{#1}}

\newcommand{\Cplx}{\mathbb{C}}
\newcommand{\N}{\mathbb{N}}
\newcommand{\R}{\mathbb{R}}

\newcommand{\mL}{\mrm{L}}
\newcommand{\mH}{\mrm{H}}

\newcommand{\mX}{\mathscr{S}}

\renewcommand{\dim}{\mrm{dim}}

\usetikzlibrary{angles,quotes}

\newtheorem{theorem}{Theorem}[section]

\newtheorem{remark}[theorem]{Remark}
\newtheorem{definition}[theorem]{Definition}

\newtheorem{proposition}[theorem]{Proposition}

\begin{document}

~\vspace{-0.6cm}
\begin{center}

{\sc \bf\fontsize{20}{20}\selectfont 
Eigenvalue falls in thin broken quantum strips} 
 
\end{center}

\begin{center}
\textsc{Lucas Chesnel}$^1$, \textsc{Sergei A. Nazarov}$^{2,3}$\\[16pt]
\begin{minipage}{0.86\textwidth}
{\small
$^1$ Inria, Ensta Paris, Institut Polytechnique de Paris, 828 Boulevard des Mar\'echaux, 91762 Palaiseau, France;\\
$^2$ St. Petersburg Department of Steklov Mathematical Institute of RAS, 27, Fontanka, 191023 St.Petersburg, Russia.\\
E-mails: \texttt{lucas.chesnel@inria.fr}, \texttt{srgnazarov@yahoo.co.uk}, \texttt{serna108@mail.ru}\\[-14pt]
\begin{center}
(\today)
\end{center}
}
\end{minipage}

\vspace{0.2cm}

\begin{minipage}{0.88\textwidth}
\noindent\textbf{Abstract.} We are interested in the spectrum of the Dirichlet Laplacian in thin broken strips with angle $\alpha$. Playing with symmetries, this leads us to investigate spectral problems for the Laplace operator with mixed boundary conditions in trapezoids of thickness $\eps>0$ small. We give an asymptotic expansion of the first eigenvalues and corresponding eigenfunctions as $\eps$ tends to zero. The new point in this work is to study the dependence with respect to $\alpha$. We highlight a curious phenomenon of diving eigenvalues: when the strip is more and more broken, at certain critical angles, that we characterize, an eigenvalue moves down very rapidly below the pack of other eigenvalues. We prove that this occurs more gently at $\alpha=0$ than at positive critical angles.\\[4pt]
\end{minipage}
\end{center}

\section{Introduction}

\begin{figure}[!ht]
\centering
\begin{tikzpicture}[scale=0.9]
\filldraw[fill=gray!20] (0,0) -- ++(8,0) -- ++(0,0.5) -- ++(-8,0)--cycle;
\draw[black,<->] (7.4,0)--(7.4,0.5);
\node at (7.4+0.2,0.25) {\small $\eps $};
\node at (0,-0.3) {\small $-1$};
\node at (8,-0.3) {\small $1$};
\node at (4,-0.3) {\small $O$};
\end{tikzpicture}\qquad 
\begin{tikzpicture}[scale=0.9]
\filldraw[fill=gray!20] (0,0) -- ++(4,0) -- ++(0,0.5) -- ++({0.5*sin(10)-4},0)--++({-4*cos(10)},{4*sin(10)})--++({-0.5*sin(10)},{-0.5*cos(10)})--cycle;
\draw[black,<->] (3.4,0)--(3.4,0.5);
\node at (3.4+0.2,0.25) {\small $\eps $};
\node at (4,-0.3) {\small $1$};
\node at (0,-0.3) {\small $O$};
\end{tikzpicture}\vspace{-0.2cm}
\caption{Thin reference (left) and broken (right) strips. \label{GeomBrokenStrip}}  
\end{figure}
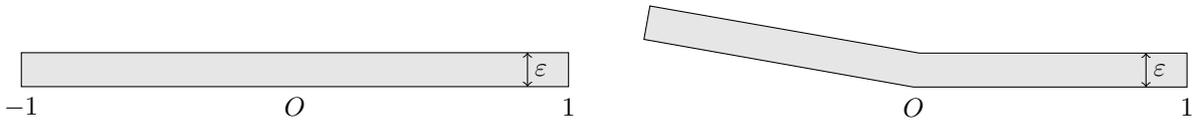

With the recent developments in the study of graphene which can be used to design devices with extraordinary properties \cite{GeNo07,Geim09,NGPNG09}, important efforts have been made to understand wave propagation phenomena in quantum waveguides in the mathematical community. In certain simple situations, this leads one to investigate the properties of the Laplace operator with Dirichlet boundary conditions in thin domains. \\
\newline
Many works concern the case where the geometry has a thin transverse section and is unbounded in some longitudinal direction. In these multiscale waveguides, because waves cannot radiate as easily as in free space, it is known that trapped modes, also called bound states, can exist. They are defined as non trivial solutions to the homogeneous problem which are of finite energy, \textit{i.e.} in the Sobolev space $\mH^1$. They have been the object of intense studies first because in general one wants to avoid their unpleasant effects. But one can also search specifically for bounded states in the continuous spectrum \cite{SaBR06,Mois09,GPRO10} because when slightly perturbed, they give rise to complex resonances with small imaginary parts in the vicinity of which one observes interesting versatile scattering phenomena (see the literature concerning the so-called Fano resonance \cite{ShTu12,ChNa18}). In this context, it has been shown in particular that as soon as a straight waveguide is bent \cite{ExSe89,DuEx95} or broken with an angle \cite{ExL,ABGM91,CLMM93}, eigenvalues appear below the continuous spectrum of the Dirichlet Laplacian.\\ 
\newline
On the other hand, some authors have focus their attention on the derivation of reduced models to describe wave propagation in latices made of thin ligaments, which are particular elements of the family of multiscale domains. This has given rise to a large body of works, among them \cite{Kuch02,Post05,MoVa07,Grie08,Post12,BeKu13,Naza14,ExKo15,NaRU15,Pank17}. 
The limit problems involve ODEs in 1D ligaments connected at some junctions points. In this field, the main difficulty lies in writing relevant transmission conditions at the nodes of the network. For the Neumann Laplacian, it has been known for a long time that Kirchhoff transmission conditions, that is continuity of the field and zero outgoing flux, must be imposed. Things are more subtle for the Dirichlet Laplacian. In general, a homogeneous Dirichlet condition must be imposed, so that information cannot propagate in the structure. But in certain particular situations, this depends on the spectral properties of the operator obtained by zooming at the junction and more precisely on the existence of so-called threshold resonances (see Definition \ref{DefTR} below), the relevant law is a transmission condition, for example Kirchhoff or anti-Kirchhoff.\\
\newline
The goal of the present article is to cross the different studies mentioned above to describe the features of the spectrum of the Dirichlet Laplacian in thin bounded broken strips characterized by a parameter $\eps>0$ small as depicted in Figure \ref{GeomBrokenStrip}. First, due to the Dirichlet boundary conditions, all eigenvalues move to $+\infty$ as $O(\eps^{-2})$ when $\eps$ tends to zero. But more originally, we prove that for a fixed small $\eps$, when the angle of the strip varies continuously, for certain values, an eigenvalue dives, \textit{i.e.} moves down very rapidly, below the normalized threshold $\pi^2/\eps^2$. \\
\newline
Concerning the analysis of the spectrum of the Dirichlet Laplacian in other thin bounded domains, one may consult \cite{KaNa00,Frei07,FrSo09,BoFr09}. Finally let us mention that the Dirichlet Laplacian in 2D domains $\om$ also appears when one is interested in the Maxwell's equations in 3D geometries of the shape $\om\times(0;c)$, $c>0$ (see \cite{GoJa92,CLMM93,BoCF25}).\\
\newline
The outline is as follows. In Section \ref{SectionSetting}, we start by defining the geometry and introduce the notation. Section \ref{SectionNumerics1} is devoted to the presentation of first numerical experiments to get some intuition of the features of the spectrum of our problem. In order to perform the asymptotic analysis as $\eps$ tends to zero, in Section \ref{SectionNFPb} we detail the properties of the near field operator obtained by zooming at $O$ and which plays a key role in the analysis. Then in Section \ref{SectionAsymptoAlpha}, we state and prove the first main result of the article, namely Theorem \ref{MainThm1}, where we deal with the asymptotics of the eigenvalues as $\eps$ tends to zero at a fixed angle $\alpha$. This allows us to identify particular angles, the critical angles, where the relevant 1D limit model consists in imposing a Neumann condition at $O$ and not a Dirichlet condition as in the generic case. We complement the numerics in Section \ref{SectionNumericalResults} and explain there how to compute these critical angles. Section \ref{SectionModelPb} contains the main originality of the paper: we derive model problems describing the transition of the spectrum for angles varying around the critical angles. The results are encapsulated in Theorems \ref{ThmModel1} and \ref{ThmModel2} where we work respectively around positive critical angles and around the null angle. In Section \ref{SectionAuxRes} we prove two auxiliary results needed in the study and finally we conclude by a few remarks, listing in particular some open questions.

\section{Setting}\label{SectionSetting}

\begin{figure}[h!]
\centering
\begin{tikzpicture}[scale=1.2]
\filldraw[fill=gray!20] (0,0) -- ++(8,0) -- ++(0,1) -- ++(-7.5,0)--cycle;
\draw[->] (0,0) -- (1,0) node[anchor=north]{$x$};
\draw[->] (0,0) -- (0,1) node[anchor=east]{$y$};
\draw[<-] (0.2,0.5) --++ (-0.6,0) node[anchor=east]{$x=y\,\tan\alpha$};
\draw[<-] (0,0) ++(90:.8) arc (90:65:.8); 
\draw[black,<->] (5,0)--(5,1);
\node at (5+0.3,0.5) { $\eps $};
\node at (0.25,0.9) { $\alpha$};
\node at (4,-0.2) { $\Sigma^\eps$};
\node at (0.6,0.5) { $\Gamma^\eps$};
\node at (4,1.2) { $\Sigma^\eps$};
\node at (8+0.3,0.5) { $\Sigma^\eps$};
\node at (8,-0.2) {\small $1$};
\node at (0,-0.2) {\small $O$};
\end{tikzpicture}
\caption{Trapezoid $\mathrm{T}^\eps$. \label{GeomTrapezoid}}
\end{figure}

\noindent Our work is motivated by the analysis of the spectrum of the Laplace operator with Dirichlet Boundary Condition (BC) in finite length thin broken strips characterized by a parameter $\eps>0$ small as displayed in Figure \ref{GeomBrokenStrip}. In particular, we are interested in describing the behavior of the first eigenvalues and corresponding eigenfunctions as $\eps>0$ tends to zero. Exploiting the symmetry of the problem, one can reduce the study to that of the spectrum of the Laplace operator with mixed BCs in the geometry of Figure \ref{GeomTrapezoid}. More precisely, for $\alpha\in(-\pi/2;\pi/2)$ and $\eps>0$, define the trapezoid
\begin{equation}\label{DefTrapezoid}
\mathrm{T}^\eps \coloneqq \{ z\coloneqq(x,y)\in \R^2\,|\,y\in(0;\eps)\mbox{ and }x\in (y\tan\alpha;1)\}.
\end{equation}
Note that for $\alpha=0$, $\mathrm{T}^\eps$ coincides with the rectangle $(0;1)\times(0;\eps)$. Let us give notation to the different components of the boundary $\partial\mathrm{T}^\eps$ of the domain (\ref{DefTrapezoid}). Set 
\[
\Sigma^\eps\coloneqq \{ z\in\partial\mathrm{T}^\eps\,|\,y = 0\mbox{ or }y=\eps\mbox{ or }x=1\}\quad \mbox{ and }\quad \Gamma^\eps \coloneqq \partial\mathrm{T}^\eps\setminus\overline{\Sigma^\eps}=\{ (y\tan\alpha,y)\,|\,y\in(0;\eps)\}.
\]
From time to time below, we will indicate the dependence with respect to $\alpha$, writing for example $\mathrm{T}^\eps(\alpha)$, $\Sigma^\eps(\alpha)$, $\Gamma^\eps(\alpha)$. We study the spectral problem with mixed BCs
\begin{equation}\label{MainProblem}
\begin{array}{|rcll}
- \Delta u &=& \lambda u & \quad\mbox{in }\mathrm{T}^\eps\\[3pt]
u &=& 0 & \quad\mbox{on }\Sigma^\eps \\[3pt]
\partial_n u  &=& 0 & \quad\mbox{on }\Gamma^\eps,
\end{array}
\end{equation}
where $\partial_n$ is the outward normal derivative on $\partial\mathrm{T}^\eps$. Note that on $\Gamma^\eps$, since $n=(-\cos\alpha,\sin\alpha)^\top$, there holds 
\[
\partial_n u=-\cos\alpha\,\partial_xu+\sin\alpha\,\partial_xu.
\]
Let us mention that the problem with Dirichlet BC on $\Gamma^\eps$ is less attractive because in this case the near field operator (see $A^\Om$ in (\ref{NearFieldOp})) has no discrete spectrum, that annuls interesting effects. Denote by $\mH_0^1(\mathrm{T}^\eps; \Sigma^\eps)$ the Sobolev space of functions in $\mH^1(\mathrm{T}^\eps)$ vanishing on $\Sigma^\eps$. Classically (see \textit{e.g.} \cite{Lad,LiMa68}), the variational formulation  of Problem (\ref{MainProblem}) writes
\begin{equation} \label{FV}
\begin{array}{|l}
\mbox{ Find }(\lambda,u)\in\R\times \mH_0^1(\mathrm{T}^\eps ; \Sigma^\eps)\setminus\{0\}\mbox{ such that}\\[5pt]
\,\dsp\int_{\mathrm{T}^\eps} \nabla u\cdot\nabla v\,dz = \lambda\,\int_{\mathrm{T}^\eps} u   v\,dz
\qquad\forall v \in \mH_0^1(\mathrm{T}^\eps ; \Sigma^\eps).
\end{array}
\end{equation}
As known for example from \cite[\S10.1]{BiSo87} or \cite[Chap.\,VIII.6]{RS78}, the variational problem (\ref{FV}) gives rise to the unbounded operator $A^\eps$ of $\mL^2(\mathrm{T}^\eps)$ such that
\begin{equation}\label{DefAeps}
\begin{array}{rccl}
A^\eps:&\mathcal{D}(A^\eps)&\quad\to&\quad\mL^2(\mathrm{T}^\eps)\\[6pt]
 & u & \quad\mapsto&\quad A^\eps u=-\Delta u,
\end{array}
\end{equation}
with $\mathcal{D}(A^\eps)\coloneqq\{u\in\mH_0^1(\mathrm{T}^\eps ; \Sigma^\eps)\,|\,\Delta u\in\mL^2(\mathrm{T}^\eps)\mbox{ and }\partial_n u=0\mbox{ on }\Gamma^\eps\}$. By applying the classical theory of regularity of solutions to elliptic problems in domains with singular geometries, see \textit{e.g.} \cite{Kond67,KoMR97}, one can characterize explicitly the domain $\mathcal{D}(A^\eps)$ (see (\ref{CalculSingularites}) for more details). The operator $A^\eps$ is positive definite and selfadjoint. Since $\mathrm{T}^\eps$ is bounded, the embedding $\mH_0^1(\mathrm{T}^\eps ; \Sigma^\eps)\subset \mL^2(\mathrm{T}^\eps)$ is compact and the spectrum of $A^\eps$ is discrete, made of the sequence of normal eigenvalues
\begin{equation}\label{DefEigenfunctions}
0<\lambda^\eps_1 < \lambda^\eps_{2} \le \dots \le \lambda^\eps_p \le \dots \quad\to +\infty,
\end{equation}
where the $\lambda^\eps_p$ are counted according to their multiplicity. Note that the Krein-Rutman theorem (see \textit{e.g.} \cite[Thm.\,1.2.5]{Henr06}) ensures that $\lambda^\eps_1$ is a simple eigenvalue. The corresponding eigenfunctions $u^\eps_p\in
\mH_0^1(\mathrm{T}^\eps ; \Sigma^\eps)$ can be subject to the orthogonality and normalization conditions 
\[
\int_{\mathrm{T}^\eps} u^\eps_p u^\eps_q\,dz =\delta_{p,q}\qquad\forall p,q\in\N^\ast\coloneqq\{1,2,3,\dots\},
\]
where $\delta_{p,q}$ is the Kronecker symbol. The general purpose of this work is to describe the behavior of the eigenpairs $(\lambda^\eps_p,u^\eps_p)$ with respect to the angle $\alpha$ for small values of $\eps$. 

\section{First numerical experiments}\label{SectionNumerics1}
In this section, we present some numerics to get an intuition of the dependence of the spectrum of $A^\eps$ with respect to $\eps$. In the case $\alpha=0$, we can use decomposition in Fourier series. For the first eigenvalues, we find  
\begin{equation}\label{CaseAlphaNull}
\lambda^\eps_p=\pi^2/\eps^2+(p+1/2)^2\pi^2,\quad p\in\N,\quad\mbox{ with }\quad u^\eps_p(x,y)=2\cos(\pi(p+1/2)x)\sin(\pi y/\eps)/\sqrt{\eps}.
\end{equation}
In particular, the eigenfunctions decompose on the fundamental mode in the transverse direction and can have several oscillations in the longitudinal direction. Notice that all eigenvalues move to $+\infty$ as $O(\eps^{-2})$ when $\eps$ tends to zero. This is expected in this thin geometry due to the Dirichlet conditions on the horizontal parts of the boundary and will be true for all $\alpha\in(-\pi/2;\pi/2)$. To get an idea of what happens for $\alpha\ne0$, we can solve numerically (\ref{MainProblem}), for example by using a finite element method. 

\begin{figure}[!ht]
\centering
\includegraphics[width=11cm]{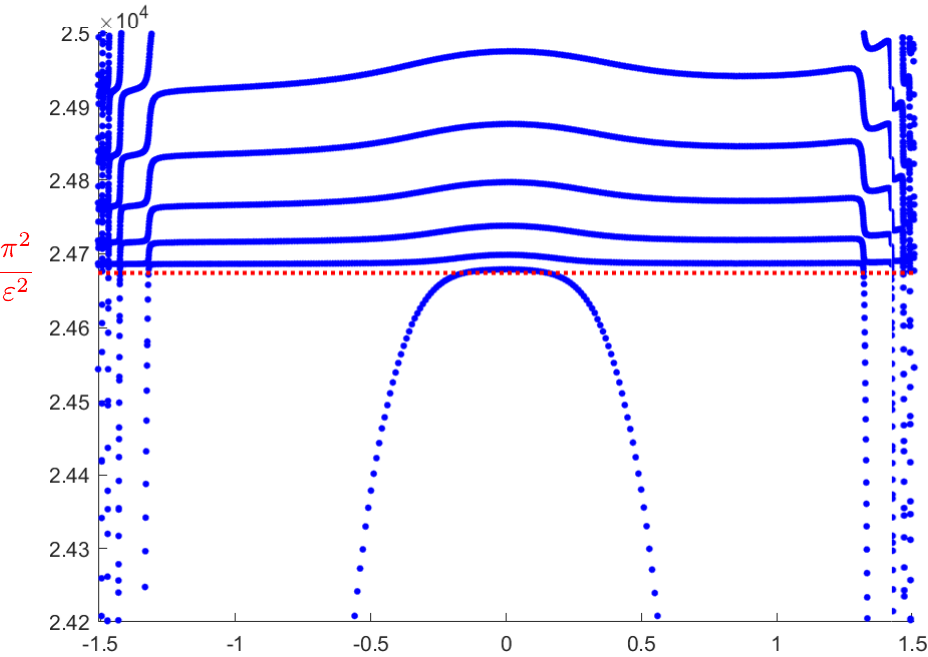}\vspace{-0.3cm}
\caption{Spectrum of $A^\eps$ with respect to $\alpha\in(-0.48\pi;0.48\pi)$ for $\eps=0.02$. The horizontal red dotted line corresponds to the normalized threshold $\pi^2/\eps^2$. \label{SpectrumVsAlpha}}
\end{figure}
\begin{figure}[!ht]
\centering
\includegraphics[width=11cm]{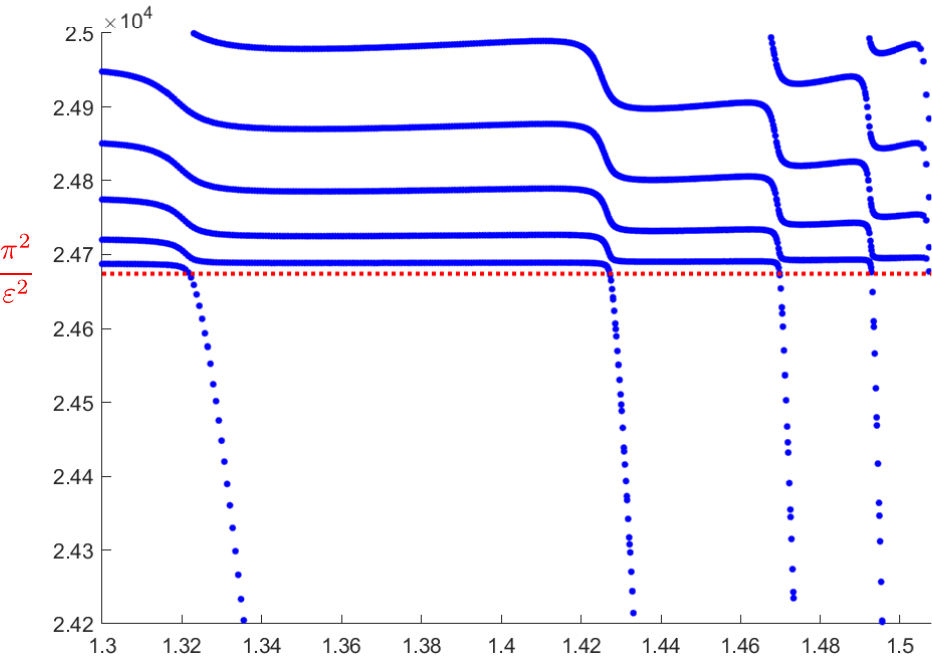}\vspace{-0.3cm}
\caption{Same quantities as in Figure \ref{SpectrumVsAlpha} with a zoom for $\alpha\in(1.3;0.48\pi)$. \label{SpectrumVsAlphaZoom}}
\end{figure}

\noindent In Figure \ref{SpectrumVsAlpha}--\ref{SpectrumVsAlphaZoom}, we display the first eigenvalues of (\ref{MainProblem}) obtained with \texttt{Freefem++} (see  \cite{Hech12}) with respect to $\alpha\in(-0.48\pi;0.48\pi)$ for a fixed small $\eps$ (we take $\eps=0.02$). First, we observe that the spectrum is approximately the same if we replace $\alpha$ by $-\alpha$. This is due to the fact that as $\eps$ tends to zero, $\mathrm{T}^\eps(\alpha)$ and $\mathrm{T}^\eps(-\alpha)$ are asymptotically congruent. For this reason, in the following we shall assume that $\alpha\ge0$. We also note that at certain particular $\alpha$, an eigenvalue dives below $\pi^2/\eps^2$, a quantity that what we will call the normalized threshold and which is marked by the horizontal red dotted line (see in particular around $\alpha=1.32$ in Figure \ref{SpectrumVsAlphaZoom}). The first goal of this work is to explain this property. Additionally, we remark that the second  eigenvalue which dives below $\pi^2/\eps^2$ when $\alpha>0$ increases, dives much more rapidly than the first one. This is also a point that we want to understand.\\ 
\newline
Finally in Figure \ref{FigEigenfunctions}, we display eigenfunctions associated with the five first eigenvalues of $A^\eps$ again for $\eps=0.02$ and for three different values of $\alpha$. Depending on $\alpha$, we see that we can have or not eigenfunctions which are localized in a neighborhood of the tip of the trapezoid. We will explain why below. 

\begin{figure}[!ht]
\centering
\includegraphics[width=15cm]{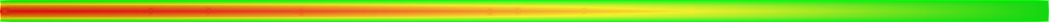}\\[-1pt]\includegraphics[width=15cm]{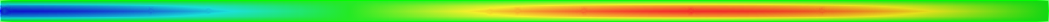}\\[-1pt]
\includegraphics[width=15cm]{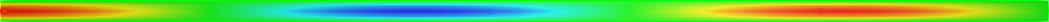}\\[-1pt]
\includegraphics[width=15cm]{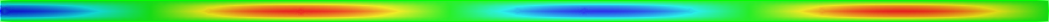}\\[-1pt]
\includegraphics[width=15cm]{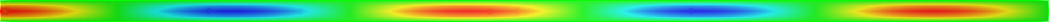}\\[12pt]
\includegraphics[width=15cm]{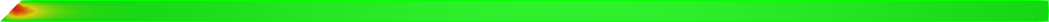}\\[-1pt]\includegraphics[width=15cm]{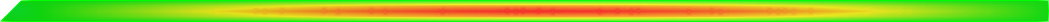}\\[-1pt]
\includegraphics[width=15cm]{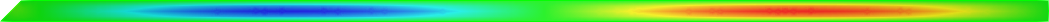}\\[-1pt]
\includegraphics[width=15cm]{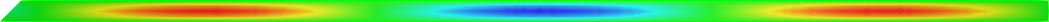}\\[-1pt]
\includegraphics[width=15cm]{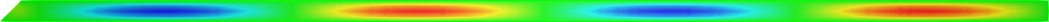}\\[12pt]
\includegraphics[width=15cm]{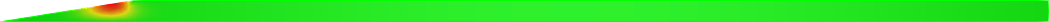}\\[-1pt]\includegraphics[width=15cm]{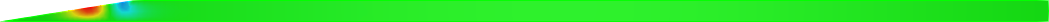}\\[-1pt]
\includegraphics[width=15cm]{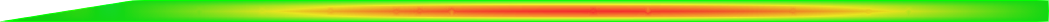}\\[-1pt]
\includegraphics[width=15cm]{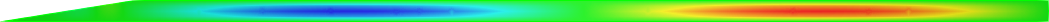}\\[-1pt]
\includegraphics[width=15cm]{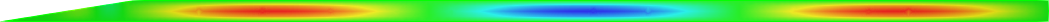}
\caption{Eigenfunctions associated with the 5 first eigenvalues of $A^\eps$ for $\eps=0.02$. Top: $\alpha=0$, middle: $\alpha=\pi/4$, bottom: $\alpha=0.45\pi\approx 1.41$.\label{FigEigenfunctions}}
\end{figure}

\noindent Our analysis will be based on the derivation of an asymptotic expansion of the eigenvalues and eigenfunctions of $A^\eps$ as $\eps$ tends to zero, allowing us to explicit the dependence with respect to the angle $\alpha$. To proceed, following the articles concerning dimension reduction in almost 1D quantum graphs, see \textit{i.e.} \cite{Grie08,Naza14,ChNa23}, we shall work with the technique of matched asymptotic expansions. Roughly speaking, we will propose different representations for the eigenfunctions of $A^\eps$, far and near the origin, that we will match in some intermediate region. In the next section, we present several objects that will be useful to define the near field expansion of the eigenfunctions of $A^\eps$.

\section{Near field operator}\label{SectionNFPb}

\begin{figure}[!ht]
\centering
\begin{tikzpicture}
\filldraw[fill=gray!20,draw=none] (0,0) -- ++(6,0) -- ++(0,2) -- ++(-5,0)--cycle;
\draw (6,0) -- ++(-6,0) -- ++(1,2) -- ++(5,0);
\draw[dashed] (6,0)-- ++(0,2);
\draw[->] (0,0) -- (1,0) node[anchor=north]{$X$};
\draw[->] (0,0) -- (0,1) node[anchor=east]{$Y$};
\draw[<-] (0.6,1.4) --++ (-1.2,0) node[anchor=east]{$X=Y\,\tan\alpha$};
\draw[<-] (0,0) ++(90:.6) arc (90:65:.6); 
\node at (0.2,0.8) {$\alpha$};
\draw[black,<->] (5,0)--(5,2);
\node at (5+0.3,1) { $1$};
\node at (3,-0.3) { $\Sigma$};
\node at (0.9,1) { $\Gamma$};
\node at (3,2.2) { $\Sigma$};
\end{tikzpicture}
\caption{Near field geometry $\Om$.}\label{GeomInnerField}
\end{figure}
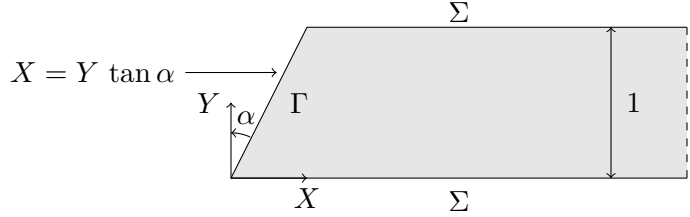

\noindent In a neighborhood of the origin, we expect the eigenfunctions of $A^\eps$ to exhibit fast variations. To catch them, we will work with the rapid variables
\[
z/\eps=(x/\eps,y/\eps),
\]
which allow us to zoom at $O$. Then, as $\eps$ tends to zero,
we are led to consider the unbounded geometry 
\[
\Om \coloneqq \{ Z\coloneqq(X,Y)\in \R^2\,|\,Y\in(0;1)\mbox{ and }X\in (Y\tan\alpha;+\infty)\}
\]
(see Figure \ref{GeomInnerField}) and examine the spectral problem 
\begin{equation}\label{PbSpectralZoom}
\begin{array}{|rcll}
-\Delta U&=&\mu U &\mbox{ in }\Om\\[3pt]
U&=&0&\mbox{ on } \Sigma\coloneqq\{ Z\in\partial\Om\,|\,Y = 0\mbox{ or }Y=1\} \\[3pt]
\partial_n U  &=& 0 & \mbox{ on }\Gamma\coloneqq\partial\Om\setminus\overline{\Sigma}.
\end{array}
\end{equation}
Note that (\ref{PbSpectralZoom}) is independent of $\eps$. Let $\mH_0^1(\Om ; \Sigma)$ stand for the Sobolev space of functions of $\mH^1(\Om)$ vanishing on $\Sigma$. We denote by $A^\Om$ the unbounded, positive definite, selfadjoint operator associated with (\ref{PbSpectralZoom}) defined in the Hilbert space $\mL^2(\Om)$, with the domain 
\begin{equation}\label{NearFieldOp}
\mathcal{D}(A^\Om)\coloneqq\{U\in\mH_0^1(\Om ; \Sigma)\,|\,\Delta U\in\mL^2(\Om)\mbox{ and }\partial_n U=0\mbox{ on }\Gamma\}.
\end{equation}
As for $A^\eps$ in (\ref{DefAeps}), one can explicit the singularity of the functions of $\mathcal{D}(A^\Om)$ at $O$ by using the theory of regularity of solutions to elliptic problems in domains with singular geometries (see the calculus in (\ref{CalculSingularites})). The continuous spectrum of $A^\Om$ coincides with the ray $[\pi^2;+\infty)$. When $\alpha = 0$ (straight end), classically one shows that the discrete spectrum of $A^\Om$, denoted by $\sigma_{\mrm{d}}(A^\Om)$, is empty. More precisely, for all $X>0$, one has the 1D Poincar\'e-Friedrichs inequality 
\[
\pi^2\int_{0}^1 U^2(X,Y)\,dY \le \int_{0}^1 (\partial_YU(X,Y))^2\,dY,\qquad\forall U\in\mH_0^1(\Om ; \Sigma)\cap\mathscr{C}^\infty(\overline{\Om}).
\]
Integrating it with respect to $X\in(0;+\infty)$, and using the density of $\mH_0^1(\Om ; \Sigma)\cap\mathscr{C}^\infty(\overline{\Om})$ in $\mH_0^1(\Om ; \Sigma)$, we obtain
\[
\pi^2\int_{\Om} U^2\,dZ \le \int_{\Om} |\nabla U|^2\,dZ,\qquad\forall U\in\mH_0^1(\Om ; \Sigma).
\]
From the minimum principle (see \textit{e.g.} \cite[Thm.\,10.2.1]{BiSo87}), we deduce that $\sigma_{\mrm{d}}(A^\Om)=\emptyset$ because the Rayleigh quotient cannot be smaller that $\pi^2$. On the other hand, for all $\alpha \in (0;\pi/2)$, it is known from \cite{ABGM91} that $A^\Om$ has at least one eigenvalue below its continuous spectrum (see also \cite{Naza12a} for more general shapes and dimensions). Using that extending $\Om$ by reflection with respect to $\Gamma$ provides a V-shaped domain, we can exploit the results from \cite{NaSh14} (see also \cite{DaRa12} as well as the amendments in \cite{Naza14c}) to get information on $\mu_1$, the smallest eigenvalue of $A^\Om$, as well as on the multiplicity of the discrete spectrum. In particular, in these works it is shown that the function $\alpha\mapsto \mu_1(\alpha)$ is smooth and strictly decreasing on $(0;\pi/2)$. Additionally, we have 
\begin{equation}\label{defAlpha}
\lim_{\alpha \to 0^+ } \mu_1(\alpha) =  \pi^2,\qquad
\lim_{\alpha \to (\pi/2)^-} \mu_1(\alpha) =  \frac{\pi^2}{4}\qquad\mbox{ and }\qquad\lim_{\alpha \to (\pi/2)^-} \#(\sigma_{\mrm{d}}(A^\Om))=+\infty,
\end{equation}
where $\#(\cdot)$ stands for the cardinality of a set. The third limit of (\ref{defAlpha}) means that the multiplicity of the discrete spectrum of $A^\Om$ can be made as large as desired for very sharp tips.
According to \cite{Grie08}, the properties of (\ref{PbSpectralZoom}) with $\mu$ coinciding with the threshold of the continuous spectrum of $A^\Om$ are determinant in the asymptotics of the eigenvalues of $A^\eps$ above the normalized threshold $\pi^2/\eps^2$ as $\eps\to0$ (see after (\ref{NearField})). This leads us to consider the problem 
\begin{equation}\label{NearFieldPb}
\begin{array}{|rcll}
\Delta U+\pi^2U&=&0&\mbox{ in }\Om\\
U&=&0&\mbox{ on } \Sigma\\[3pt]
\partial_n U  &=& 0 & \mbox{ on }\Gamma.
\end{array}
\end{equation}
To study (\ref{NearFieldPb}), adapted objects must be introduced. First, define the waves $w_0$, $w_1$ such that
\[
w_0(Z)=\sin(\pi Y),\qquad\quad  w_1(Z)=X\sin(\pi Y),
\]
which solve the equations of (\ref{NearFieldPb}) for $X>\tan\alpha$. The waves $w_0$, $w_1$ are somehow degenerated because their flux of energy through a transverse section is zero. More precisely, for $L>\tan\alpha$, if we set $\Gamma_L\coloneqq\{L\}\times(0;1)$, for $j=0,1$, one has
\[
\Im m\,\left(\int_{\Gamma_L}\partial_Xw_j\overline{w_j}\,dY\right)=0.
\]
Following \cite[Chap.\,5]{NaPl94}, this degeneration leads us to define the linear wave packets $w^{\mrm{out}}$, $w^{\mrm{in}}$ such that
\begin{equation}\label{defModeswinout}
\begin{array}{|lcl}
w^{\mrm{out}}(Z) &\hspace{-0.2cm}=& \hspace{-0.2cm}w_1(Z)- iw_0(Z) =(X- i)\sin(\pi Y)\\[5pt]
w^{\mrm{in}}(Z) &\hspace{-0.2cm}=& \hspace{-0.2cm}w_1(Z)+ iw_0(Z) =(X+ i)\sin(\pi Y).
\end{array}
\end{equation}
Here the notation ``in/out'' stands for ``incoming/outgoing''. Indeed, one finds
\[
\Im m\,\left(2\int_{\Gamma_L}\partial_Xw^{\mrm{out}}\,\overline{w^{\mrm{out}}}\,dY\right)=1,\qquad\quad \Im m\,\left(2\int_{\Gamma_L}\partial_Xw^{\mrm{in}}\,\overline{w^{\mrm{in}}}\,dY\right)=-1.
\]
Then the theory presented in \cite[Chap.\,5]{NaPl94} guarantees that Problem (\ref{NearFieldPb}) admits a solution with the expansion
\begin{equation}\label{ScaSol}
W  = w^{\mrm{in}}+\mathbb{S}\,w^{\mrm{out}}+\widetilde{W},
\end{equation}
where $\mathbb{S}\in\Cplx$ and $\widetilde{W}$ decays exponentially at infinity. In this simple situation the so-called threshold scattering matrix $\mathbb{S}$ becomes a scalar and the fact that $\mathbb{S}$ is unitary translates into the relation 
\begin{equation}\label{Sunitary}
|\mathbb{S}|=1.
\end{equation}
In other words, $\mathbb{S}$ is located on the unit circle in the complex plane. Identity (\ref{Sunitary}) is simply obtained by taking the imaginary part and the limit $L\to+\infty$ in the equation
\[
0=\int_{\Om_L}(\Delta W+\pi^2 W)\overline{W}\,dZ=\int_{\Om_L}-|\nabla W|^2+\pi^2|W|^2\,dZ+\int_{\Gamma_L}\partial_XW\,\overline{W}\,dY
\]
where $\Om_L\coloneqq\{Z\in\Om\,|\,X<L\}$. It may happen, we do not know, that (\ref{NearFieldPb}) also admits solutions which are exponentially decaying at infinity (trapped modes). In this situation, the function (\ref{ScaSol}) is not uniquely defined. However these trapped modes do not modify the value of $\mathbb{S}$ because they decay exponentially at infinity. Following \cite{MoVa08,Naza17bis,KoNS16,Naza20Threshold}, now we introduce some vocabulary. Define the vector spaces
\begin{equation}\label{DefSpaceX}
\begin{array}{lcl}
\mX&\hspace{-0.2cm}\coloneqq&\hspace{-0.2cm}\dsp\{U=C_1 w_1+C_0 w_0+\widetilde{U}\mbox{ satisfying  }(\ref{NearFieldPb})\mbox{ with }C_0,\,C_1\in\Cplx\mbox{ and }\widetilde{U}\in\mH^1(\Om)\};\\[10pt]
\mX_{\mrm{bo}}&\hspace{-0.2cm}\coloneqq&\hspace{-0.2cm}\dsp\{U= C_0 w_0+\widetilde{U}\mbox{ satisfying  }(\ref{NearFieldPb})\mbox{ with }C_0\in\Cplx\mbox{ and }\widetilde{U}\in\mH^1(\Om)\};\\[10pt]
\mX_{\mrm{tr}}&\hspace{-0.2cm}\coloneqq&\hspace{-0.2cm}\dsp\{U\in\mH^1(\Om)\mbox{ satisfying  }(\ref{NearFieldPb})\}.
\end{array}\hspace{-0.3cm}
\end{equation}
Here $\mX_{\mrm{bo}}$ is the space of bounded solutions of (\ref{NearFieldPb}) while $\mX_{\mrm{tr}}$ stands for the space of trapped modes. Generally speaking, the existence of trapped modes at the threshold is in some sense a rare phenomenon. In particular, it is unstable with respect to perturbations of the geometry. Therefore we think that there holds  $\mX_{\mrm{tr}}=\{0\}$ for most $\alpha$ though it is an open problem to prove such result. On the other hand, we do not know if we can have $\mX_{\mrm{tr}}\ne\{0\}$ for certain values of $\alpha$.

\begin{definition}\label{DefTR}
We say that $A^{\Om}$ has a Threshold Resonance (TR) if $\mX_{\mrm{bo}}\ne\{0\}$, i.e. if (\ref{NearFieldPb}) admits a non zero bounded solution.
\end{definition}

\begin{definition}
We say that $A^{\Om}$ has a proper TR if the quotient space
\[
\mX_{\dagger}\coloneqq\mX_{\mrm{bo}}/\mX_{\mrm{tr}}
\]
contains a non trivial element. This is equivalent to say that $A^{\Om}$ has a proper TR if (\ref{NearFieldPb}) admits a bounded solution which does not decay at infinity. If $\mX_{\mrm{bo}}=\mX_{\mrm{tr}}\ne\{0\}$, we say that the TR is improper. 
\end{definition}
\noindent The quotient space $\mX_{\dagger}$ is sometimes called the space of almost standing waves of (\ref{NearFieldPb}). 
From the definition of $\mX_{\dagger}$ in  (\ref{DefSpaceX}), we see that one has always $\dim\,\mX_{\dagger}\le1$. The following statement (see \cite[Thm.\,7.1]{Naza16} or \cite{Naza14} in more general contexts) provides a characterization of the dimension of $\mX_{\dagger}$. 
\begin{proposition}\label{PropoCharacterisation}
There holds
\[
\begin{array}{|ll}
\mX_{\dagger}=\{0\}&\mbox{ if }\mathbb{S}\ne-1\\[2pt]
\dim\,\mX_{\dagger}=1 &\mbox{ if }\mathbb{S}=-1.
\end{array}
\]
\end{proposition}
\begin{proof}
When $\mathbb{S}=-1$, the function $W$ introduced in (\ref{ScaSol}) decomposes as
\begin{equation}\label{ProofBounded}
W  = w^{\mrm{in}}-w^{\mrm{out}}+\widetilde{W}=2iw_0+\widetilde{W}.
\end{equation}
Therefore it is an element of $\mX_{\mrm{bo}}\setminus\mX_{\mrm{tr}}$, which shows that $\dim\,\mX_{\dagger}=1$. On the other hand, using energy considerations, one proves that one has always $\dim\,(\mX/\mX_{\mrm{tr}})=1$. When $\mathbb{S}\ne-1$, $W$ belongs to $\mX\setminus\mX_{\mrm{bo}}$. This implies that $\mX_{\dagger}=\mX_{\mrm{bo}}/\mX_{\mrm{tr}}=\{0\}$ in this case. 
\end{proof}

\noindent The particular values of $\alpha$ such that $\mathbb{S}(\alpha)=-1$ will play an important role in the sequel. When $\alpha$ varies in $[0;\pi/2)$, according to 
(\ref{Sunitary}), $\mathbb{S}(\alpha)$ runs on the unit circle in the complex plane and from time to time hits the value $-1$. More precisely, one can show first that the map $\alpha\mapsto \mathbb{S}(\alpha)$ is continuous on $[0;\pi/2)$. To proceed, quite classically, one can work with smooth diffeomorphisms to transform small (geometrical) changes of $\alpha$ into small linear perturbations of an operator defined in a fixed domain. Then one can apply usual results of the perturbation theory for linear operators, see \textit{e.g.} \cite[Chap.\,7]{Kato95}, \cite[Chap.\,10]{BiSo87}, \cite[Chap.\,XII]{RS78}. Besides, we have the following statement whose proof is postponed to Section \ref{SectionAuxRes}.
\begin{proposition}\label{PropoDiscrete}
Assume that $\alpha^\star\in(0;\pi/2)$ is such that $\mathbb{S}(\alpha^\star)=-1$. Then there holds
\[
\cfrac{\partial\mathbb{S}}{\partial\alpha}(\alpha^\star)\ne0.
\]
\end{proposition} 
\noindent This result ensures that the set of values of $\alpha\in[0;\pi/2)$ such that $\mathbb{S}(\alpha)=-1$ is discrete. We know that this set is non-empty because it contains the value $0$. Indeed, for $\alpha=0$, we simply have 
\[
W(Z)= w^{\mrm{in}}-w^{\mrm{out}}=2i\sin(\pi Y)
\]
and so $\mathbb{S}(0)=-1$. Let us denote by 
\begin{equation}\label{DefAlphaP}
0=\alpha_0^\star<\alpha_1^\star<\dots<\alpha_k^\star<\dots
\end{equation}
the sequence of values of $\alpha$ such that $\mathbb{S}(\alpha^\star_k)=-1$. In the following, these $\alpha_k^\star$ will be called the critical angles. We will assume that there exist positive critical angles, a property that unfortunately we have not been able to prove. More precisely, the increase of the multiplicity of the discrete spectrum of $A^\Om$ as $\alpha$ tends to $\pi/2$ is due to detaching of an eigenvalue from the threshold $\pi^2$ at a TR. However we cannot prove that these TRs are proper. Let us emphasize that the numerics of Section \ref{SectionNumericalResults} (see in particular Figure \ref{TRPhase}) seems to indicate that positive critical angles indeed exist.

\section{Asymptotics at a fixed $\alpha$}\label{SectionAsymptoAlpha}

Now we can describe the asymptotic analysis of the eigenvalues and eigenfunctions of $A^\eps$ as $\eps$ tends to zero for a given $\alpha\in[0;\pi/2)$. For the sake of conciseness, we shall only give the main ideas.\\
\newline
Denote by 
\begin{equation}\label{DefEigenValDS}
0<\mu_1<\mu_2\le \mu_3 \le \cdots \le \mu_{N_\circ}<\pi
\end{equation}
the eigenvalues of the discrete spectrum of $A^\Om$. Set also  
\begin{equation}\label{DefNbEigTR}
N_\dagger\coloneqq\dim\,\mX_{\mrm{tr}}\in\N.
\end{equation}
Observe that the quantities $\mu_{N_\circ}$, $N_\dagger$ depend on $\alpha$ but we do not write this dependence explicitly here. In the particular case $\alpha=0$ (straight end), as already said, we have $N_\circ=N_\dagger=0$. On the other hand, there holds $N_\circ\ge1$ when $\alpha\in(0;\pi/2)$.

\begin{theorem}\label{MainThm1}
Pick $\alpha\in[0;\pi/2)$. For $p\in\N^\ast$, there are positive constants $C_p$, $\delta_p$, $\eps_p$ such that the eigenvalue $\lambda^\eps_p$ of $A^\eps$ introduced in (\ref{DefEigenfunctions}) satisfies the following estimate:\\[6pt]
$\bullet$ $\mbox{If }p\in\{1,\dots,N_\circ\}:$
\begin{equation}\label{result_thm_1}
\big|\lambda^\eps_p-\eps^{-2}\mu_p\big| \le C_p\,e^{-\delta_p\sqrt{\pi^2-\mu_p}/\eps}\qquad\forall\eps\in(0;\eps_p];\\[4pt]
\end{equation}
$\bullet$ $\mbox{If }p\in\{N_\circ+1,\dots,N_\circ+N_\dagger\}:$
\[
\big|\lambda^\eps_p-\eps^{-2}\pi^2\big| \le C_p\,e^{-\delta_p\pi\sqrt{3}/\eps}\qquad\forall\eps\in(0;\eps_p];\\[4pt]
\]
$\bullet$ $\mbox{If }p=N_\circ+N_\dagger+q,\,q\in\N^{\ast}:$
\begin{equation}\label{result_thm_3}
\begin{array}{ll}
\phantom{i}i)\mbox{ if }\alpha\ne\alpha_k^\star\mbox{ for all }k\in\N, & \big|\lambda^\eps_p-(\eps^{-2}\pi^2+q^2\pi^2)\big|\le C_p\,\eps^{\delta_p}\qquad\forall\eps\in(0;\eps_p];\\[10pt]
\phantom{}ii)\mbox{ if }\alpha=\alpha_k^\star\mbox{ for some }k\in\N, &  \big|\lambda^\eps_p-(\eps^{-2}\pi^2+(q-1/2)^2\pi^2)\big|\le C_p\,\eps^{\delta_p}\qquad\forall\eps\in(0;\eps_p].
\end{array}
\end{equation}
Here the quantities $\mu_p$, $N_\circ$ (resp. $N_\dagger$) are introduced in (\ref{DefEigenValDS}) (resp. (\ref{DefNbEigTR})) while the $\alpha^\star_k$ are the critical angles defined in (\ref{DefAlphaP}).
\end{theorem}
\noindent Let us comment this statement. First, it indicates that the asymptotics of the first eigenvalues of $A^\eps$ is directly dictated by the discrete spectrum of $A^\Om$. Second, the behavior of the next eigenvalues of $A^\eps$ depends on the geometry of $\Om$ and more specifically on the existence or absence of proper TR for $A^\Om$. Note that $C_p$ and $\eps_p$ depend on $\delta_p$. Moreover $\eps_p$, $C_p$ tend respectively to $0$, $+\infty$ as $p$ tends to $+\infty$.
\begin{proof}
Let us start with the asymptotics of the $\lambda^\eps_p$ for $p\le N_\circ$. Let $u^\eps_p$ be an eigenfunction associated with $\lambda^\eps_p$. As a first approximation when $\eps\to0$, it is natural to consider the expansions
\begin{equation}\label{DefAn}
\lambda^\eps_p=\eps^{-2}\mu_p+\dots,\qquad u^\eps_p=v(z/\eps)+\dots
\end{equation}
where $\mu_p\in(0;\pi^2)$ stands for an eigenvalue of the discrete spectrum of $A^{\Om}$ introduced in (\ref{DefEigenValDS}) and $v\in \mathcal{D}(A^\Om)$ is a corresponding eigenfunction. More precisely, using decomposition in Fourier series, for $v$ one obtains the representation, for $Z$ such that $X>\tan\alpha$,
\[
v(Z)=\sum_{m=1}^\infty a_m\,e^{-\sqrt{m^2\pi^2-\mu_p}X}\sin(m\pi Y),
\]
for some constants $a_m\in\R$, $m\in\N^\ast$. As a consequence, inserting the pair $(\eps^{-2}\mu_p,v(\cdot/\eps))$ in Problem (\ref{MainProblem}) only leaves a small discrepancy in $O(e^{-\sqrt{\pi^2-\mu_p}/\eps})$ on the boundary at $x=1$, which can be exploited to establish the error estimate (\ref{result_thm_1}). 
This asymptotic expansion explains the localization effect observed in Figure \ref{FigEigenfunctions} for the eigenfunctions associated with the first eigenvalue (resp. two first eigenvalues) in the case $\alpha=\pi/4$ (resp. $\alpha=0.45\pi$). The analysis is similar to describe the behavior of the $\lambda^\eps_p$, $p=N_\circ+1,\dots,N_\circ+N_\dagger$. \\
\newline
Now we take $p=N_\circ+N_\dagger+q$ for some $q\in\N^{\ast}$. Again, let $u^\eps_p$ be an eigenfunction associated with $\lambda^\eps_p$. To simplify the notation, let us remove the subscript ${}_p$. To capture the different scales of variations of $u^\eps$ in the domain $\mathrm{T}^\eps$, we use the method of matched asymptotic expansions (see the monographs \cite{VD,Ilin}, \cite[Chap.\,2]{MaNaPl} and others). Here it consists in working with two expansions of the field: one far from the tip with respect to a slow variable and another one with respect to a rapid variable $z/\eps$ in a small neighborhood of the slanted side. These expansions are then matched in some intermediate region to obtain a global representation of the eigenfunctions. As an approximation when $\eps\to0$ far from the tip, we consider the ans\"atze 
\begin{equation}\label{ansatzvp}
\lambda^\eps=\eps^{-2}\pi^2+\eta+\dots
\end{equation}
with 
\begin{equation}\label{ExpansionTR0_2}
u^\eps(z)=\gamma(x)\sin(\pi y/\eps)+\dots\qquad\mbox{(far field expansion)}.
\end{equation}
Here the constant $\eta$ and the function $\gamma$ are to be determined. Inserting (\ref{ansatzvp})--(\ref{ExpansionTR0_2}) into Problem (\ref{MainProblem}), we obtain 
\begin{equation}\label{farfield}
\begin{array}{|rcl}
\partial^2_x\gamma+\eta\gamma&=&0\quad\mbox{ in }I\coloneqq(0;1)\\[3pt]
\gamma(1)&=&0.
\end{array}
\end{equation}
To define $\gamma$ completely, we need to complement (\ref{farfield}) with some condition at the origin. To derive it, we match the behavior of $\gamma$ with the one of some near field expansion of $u^\eps$. More precisely, in a neighborhood of the origin we look for an expansion of $u^\eps$ in the form 
\begin{equation}\label{NearField}
u^\eps(z)=U^0(z/\eps)+\eps U'(z/\eps)+\dots \qquad\mbox{(near field expansion)}
\end{equation}
with $U^0$, $U'$ to be determined. Inserting (\ref{NearField}) and (\ref{ansatzvp}) in (\ref{MainProblem}), we find that $U^0$, $U'$ must satisfy the near field problem (\ref{NearFieldPb}) at the threshold value of the continuous spectrum of $A^\Om$. Let us look for $U^0$, $U'$, in the form 
\[
U^0=c^0W,\qquad \quad U'=c'W,\qquad
\]
where $c^0,\,c'\in\R$ are constant to set and $W$ is the function introduced in (\ref{ScaSol}). For the far field expansion, at the origin we have the Taylor series
\begin{equation}\label{TaylorFarField0}
\begin{array}{rcl}
\gamma(x)\sin(\pi y/\eps) &= &(\gamma(0)+x\partial_x\gamma(0)+\dots)\sin(\pi y/\eps) \\[4pt]
&= &(\gamma(0)+\eps (x/\eps)\partial_x\gamma(0)+\dots)\sin(\pi y/\eps).
\end{array}
\end{equation}
On the other hand, in the near field expansion, when $x/\eps\to+\infty$, one has
\begin{equation}\label{NearField0}
\begin{array}{rcl}
W(z/\eps)&=&\left((x/\eps)+i)+\mathbb{S}((x/\eps)-i)\right)\sin(\pi y/\eps)+\dots\\[4pt]
&=&\left(i(1-\mathbb{S})+(x/\eps)(1+\mathbb{S})\right)\sin(\pi y/\eps)+\dots
\end{array}
\end{equation}
By matching (\ref{ExpansionTR0_2}) and (\ref{NearField}) in some intermediate region where $x\to0$ and $x/\eps\to+\infty$ as $\eps\to0$ (for example in the zone $\{z\in \mathrm{T}^\eps_\tau\,|\,\sqrt{\eps}<x<2\sqrt{\eps}\}$), exploiting (\ref{TaylorFarField0}) and (\ref{NearField0}), we find that we must impose
\begin{equation}\label{Matching1}
\gamma(0)=c^0i(1-\mathbb{S}),\qquad\quad 0=c^0(1+\mathbb{S}),\qquad\quad \partial_x\gamma(0)=c'i(1+\mathbb{S}).
\end{equation}
At this stage, we have to divide the analysis according to the value of $\mathbb{S}$.\\
\newline
Consider first the generic situation $\alpha\ne\alpha_k^\star$ for all $k\in\N$, \textit{i.e.} $\mathbb{S}\ne-1$. In that case, the second relation of (\ref{Matching1}) leads us to set $c^0=0$.
 Thus we obtain $U^0\equiv0$ and from the first relation of (\ref{Matching1}), we find that we must impose the Dirichlet condition 
\begin{equation}\label{farfield_BC}
\gamma(0)=0
\end{equation}
to complement the 1D model problem (\ref{farfield}). Solving the spectral problem (\ref{farfield}), (\ref{farfield_BC}), we obtain 
\begin{equation}\label{DefModel1DD}
\eta=q^2\pi^2 \quad\mbox{ for }q\in\N^\ast,\qquad\qquad  \gamma(x)=C \sin(q\pi x),
\end{equation}
where $C$ is a normalization factor. This establishes the item $i)$ of (\ref{result_thm_3}). Note that here, contrary to above, the eigenfunctions are not localized at the tip of the trapezoid $\mathrm{T}^\eps$. This is what we observe in Figure \ref{FigEigenfunctions}. There we also find that the shape of the eigenfunctions associated with the eigenvalues above the normalized threshold $\pi^2/\eps^2$ is in agreement with the expansion (\ref{ExpansionTR0_2}) when $\alpha=\pi/4$ and $\alpha=0.45\pi$. In particular, we indeed note at the tip a behavior in $x$ corresponding to a Dirichlet BC at the origin.\\
\newline
Finally, we consider the particular situation where $\alpha$ coincides with a critical angle, \textit{i.e.} $\alpha=\alpha_k^\star$ for some $k\in\N$, so that $\mathbb{S}=-1$. In that case, from the third relation of (\ref{Matching1}), we infer that we must complement the 1D model problem (\ref{farfield}) with the Neumann BC
\begin{equation}\label{farfield_BC_N}
\partial_x\gamma(0)=0.
\end{equation}
Solving the spectral problem (\ref{farfield}), (\ref{farfield_BC_N}), we get 
\[
\eta=(q-1/2)^2\pi^2 \quad\mbox{ for }q\in\N^\ast,\qquad\qquad  \gamma(x)=C \cos(\pi(q-1/2)x),
\]
where $C$ is a normalization factor. Then, from the first relation of (\ref{Matching1}), we find $c^0=-i\gamma(0)/2$ and so $U^0=-i\gamma(0)W/2$. This shows the item $ii)$ of (\ref{result_thm_3}). In the particular case $\alpha=\alpha^\star_0=0$, we recover the formulas (\ref{CaseAlphaNull}) and the higher order terms in (\ref{ansatzvp})--(\ref{ExpansionTR0_2}) are null. \\
\newline
Admittedly the above asymptotic analysis is rather formal. We emphasize that rigorous error estimates can be obtained by adapting the techniques presented for example in \cite{Post05,Grie08,Post12,Naza17,Naza18,Naza23} for similar problems. Let us describe the mechanism of the proof, which is in two steps. First, one exploits the classical lemma on ``almost eigenvalues and eigenfunctions'' (see \cite{ViLu} for the original version or, \textit{e.g.}, \cite[Lem.\,1]{Naza23}). For example, for the justification of (\ref{result_thm_1}), from the $v$ introduced in (\ref{DefAn}), one constructs some functions $v^{\mrm{app}}\in \mathcal{D}(A^\eps)$ such that
\[
\|A^\eps v^{\mrm{app}}-\eps^{-2}\mu_p\, v^{\mrm{app}}\|_{\mL^2(\mathrm{T}^\eps)} \le  C_p\,e^{-\sqrt{\pi^2-\mu_p}/\eps}.
\]
This ensures that $A^\eps$ has an eigenvalue $\lambda_{N^\eps_p}$, with an unknown $N^\eps_p$, in the segment 
\[
[\eps^{-2}\mu_p-C_p\,e^{-\sqrt{\pi^2-\mu_p}/\eps};\eps^{-2}\mu_p+C_p\,e^{-\sqrt{\pi^2-\mu_p}/\eps}].
\]
Next, one uses \textit{e.g.} the work of \cite{Grie08} or \cite{Post12} to prove that $\eps^2\lambda_p^\eps$ converges to some eigenvalue $\mu_{M^p}$ (again with some unknown $M^p$) of the limit problem. Finally, gathering the two results, one obtains $N^\eps_p=M^p=p$. 
\end{proof}

\section{Computation of the critical angles}\label{SectionNumericalResults}

Theorem \ref{MainThm1} in the previous section indicates that, as $\eps$ tends to zero, the asymptotics of the eigenvalues of $A^\eps$ above $\pi^2/\eps^2$ differs whether $\alpha$ is a critical angle or not. Here we wish to explain how to compute critical angles and show numerical evidence of existence of positive critical angles. Moreover we want to illustrate the shape of the eigenfunctions associated with the first eigenvalues of $A^\eps$ for $\alpha=\alpha^\star_1$.\\
\newline
To compute $\mathbb{S}(\alpha)$, we approximate the function $W$ introduced in (\ref{ScaSol}) which solves the near field problem (\ref{NearFieldPb}) at the threshold, by some function $\hat{W}$. Following \cite[\S6.2]{ChNa23}, \cite[\S9]{ChNa25}, this is done by truncating the domain $\Om$ at $X=L$ with $L$ large and by imposing the complex Robin condition
\begin{equation}\label{RobinConditions}
 \cfrac{\partial (\hat{W}-w^{\mrm{in}})}{\partial X}=\cfrac{1}{L-i}\,(\hat{W}-w^{\mrm{in}})\quad\mbox{ at }X= L.
\end{equation}
This leads us to consider the variational formulation
\begin{equation}\label{Pb_Approximation}
\begin{array}{|l}
\mbox{Find }\hat{W}\in\mH^1_{0}(\Om_L;\Sigma_L)\mbox{ such that for all }v\in\mH^1_{0}(\Om_L;\Sigma_L) \\[3pt]
\dsp\int_{\Om_L}\nabla \hat{W}\cdot\nabla v-\pi^2\,\hat{W} v\,dZ-\cfrac{1}{L-i}\int_{\Gamma_L} \hat{W} v\,dY=-\cfrac{2i}{L^2+1}\int_{\Gamma_L} w^{\mrm{in}} v\,dY,
\end{array}
\end{equation}
with 
\[
\Om_L=\{Z\in\Om\,|\,X<L\},\qquad \Sigma_L\coloneqq\{Z\in\partial\Om_L\,|\,Y=0\mbox{ or }Y=1\},\qquad \Gamma_L=\{L\}\times(0;1) 
\]
and $\mH^1_{0}(\Om_L;\Sigma_L)=\{v\in\mH^1(\Om_L)\,|\,v=0\mbox{ on }\Sigma_L\}$. Note that (\ref{RobinConditions}) is an approximated transparent condition: the function $w^{\mrm{out}}$ defined in (\ref{defModeswinout}) satisfies it exactly but $W-w^{\mrm{in}}$ only up to an error which decays when $L$ increases. More precisely, one can prove that $\hat{W}$ yields a good approximation of $W$ with an error which is exponentially decaying with $L$. In practice, we solve the problem (\ref{Pb_Approximation}) with a P2 finite element method thanks to \texttt{Freefem++} and we take $L=1+\tan\alpha$ to make sure that the oblique side $\Gamma$ of $\partial\Om$ does not meet $\Gamma_L$. Then replacing $W$ by $\hat{W}$ in the exact formula
\[
\mathbb{S}(\alpha)=\cfrac{2}{L^2+1}\,\int_{\Gamma_L} (W-w^{\mrm{in}})\,w^{\mrm{in}}\,dY,
\]
we get an approximation of $\mathbb{S}(\alpha)$. In Figure \ref{TRCpxPlane}, we display the values of $\mathbb{S}(\alpha)$ for $400$ values of $\alpha\in[0;0.48\pi)$. In accordance with relation (\ref{Sunitary}), the $\mathbb{S}(\alpha)$ are located on the unit circle. In Figure \ref{TRPhase}, we represent the phase of $\mathbb{S}(\alpha)$ for the same values of $\alpha$. In agreement with Assumption (\ref{DefAlphaP}), we observe a sequence of values of $\alpha=\alpha_k^\star$ such that $\mathbb{S}(\alpha_k^\star)=-1$. This sequence seems to accumulate at $\pi/2$. On these graphs, we note two interesting features. First, the phase of $\alpha\mapsto \mathbb{S}(\alpha)$ is not monotone. This makes it difficult to show that $\alpha\mapsto \mathbb{S}(\alpha)$ passes through $-1$ an infinite number of times, \textit{i.e.} that the family $(\alpha_k^\star)$ contains an infinite number of terms. Second, we observe that for certain ranges of $\alpha$, $\mathbb{S}(\alpha)$ runs rapidly on the unit circle. This is due to the presence of complex resonances with small imaginary part wandering in the complex plane close to the threshold value $\pi^2$ (not depicted here). These fast changes of the phase of $\mathbb{S}(\alpha)$ are due to the so-called Fano resonance (see \cite{ShTu12,ChNa18}).

\begin{figure}[!ht]
\centering
\includegraphics[width=10cm]{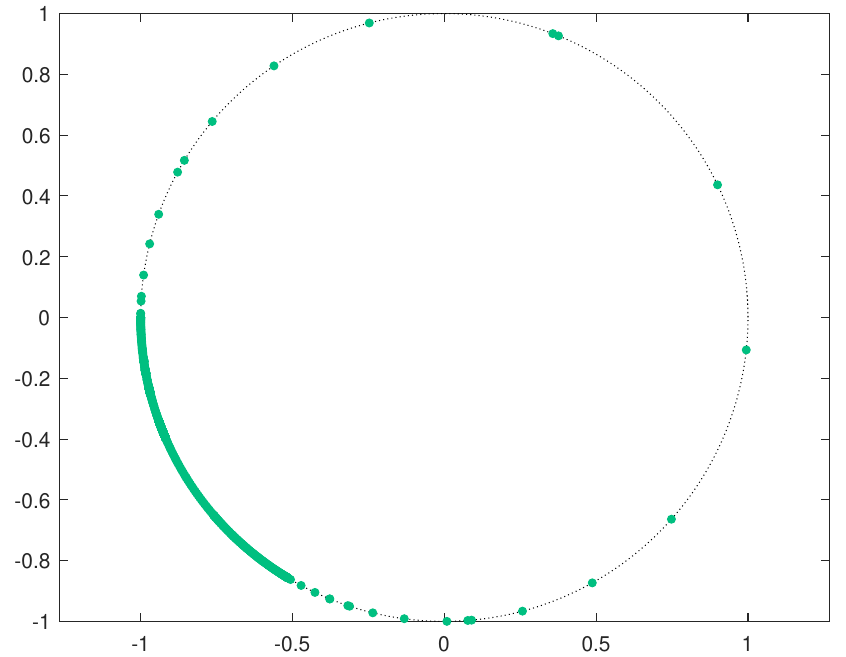}
\caption{$\mathbb{S}(\alpha)$ for $400$ values of  $\alpha\in[0;0.48\pi)$
in the complex plane.\label{TRCpxPlane}}
\end{figure}

\begin{figure}[!ht]
\centering
\includegraphics[width=8cm]{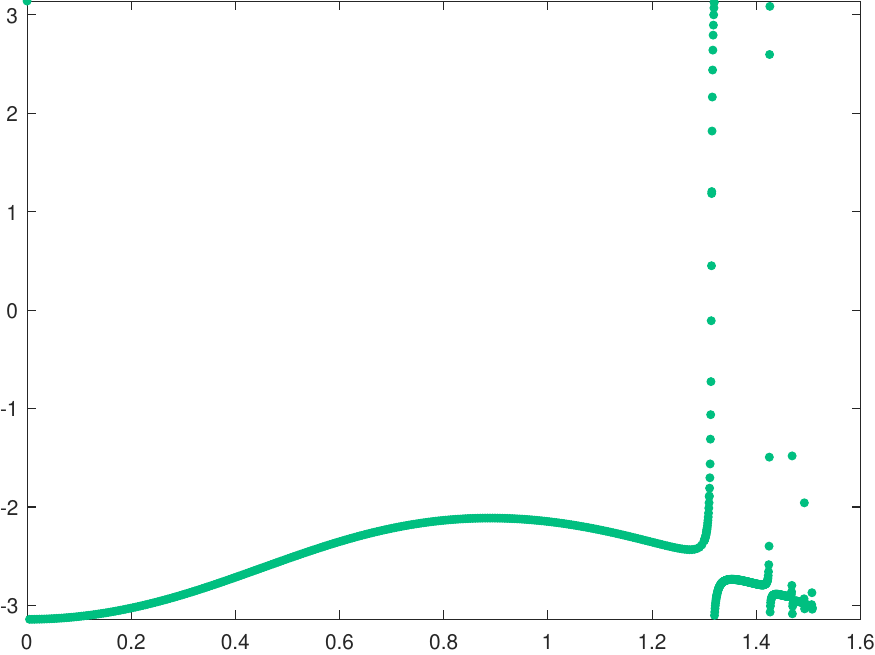}\quad\includegraphics[width=8cm]{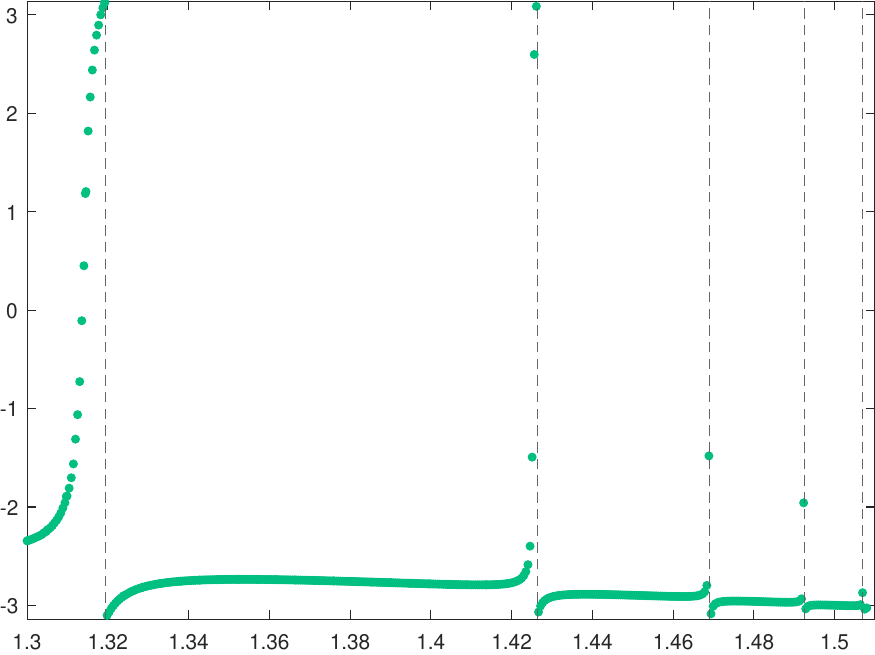}
\caption{Phase of $\mathbb{S}(\alpha)$ in $(-\pi;\pi]$  for $400$ values of  $\alpha\in[0;0.48\pi)$. The graph on the right is a zoom-in of that of the left. On the right picture, the vertical dashed lines mark the first positive critical angles.\label{TRPhase}}
\end{figure}

\noindent From these graphs, we can identify the value of $\alpha_1^\star$, namely 
\[
\alpha^\star_1\approx  1.321.
\]
Note that this corresponds up to small errors to the angles found in \cite{londergan1999binding,NaSh14} such that an eigenvalue appears below the continuous spectrum of $A^\Om$. This is not surprising because one can prove the following result: if there holds $\mathbb{S}=-1$ for a certain angle, then a small increase of $\alpha$ gives birth to an eigenvalue in the discrete spectrum of $A^\Om$ (see \cite{Naza20Threshold}).\\
\newline
In Figure \ref{FigEigenfunctionsAlphaS1}, we display the eigenfunctions associated with the five first eigenvalues of $A^\eps$ for $\eps=0.02$ and $\alpha=1.321$. As expected, we observe that only the eigenfunction associated with $\lambda_1^\eps$ is localized in a neighborhood of the tip. Moreover, the shape of the eigenfunctions corresponding to higher eigenvalues is in agreement with the expansion (\ref{ExpansionTR0_2}), (\ref{DefModel1DD}), that is 
\[
u^\eps(z)=C\cos(\pi(q-1/2)x)\sin(\pi y/\eps)+\dots. 
\]
In particular, we indeed observe at the tip a behavior in $x$ corresponding to a Neumann BC at the origin and not a Dirichlet BC as in the generic case $\alpha\ne\alpha_k^\star$, $k\in\N$ (compare with Figure \ref{FigEigenfunctions} for $\alpha\ne0$).

\begin{figure}[!ht]
\centering
\includegraphics[width=15cm]{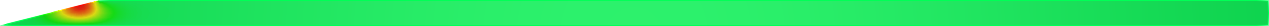}\\[-1pt]\includegraphics[width=15cm]{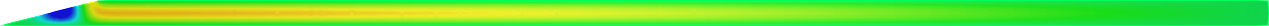}\\[-1pt]
\includegraphics[width=15cm]{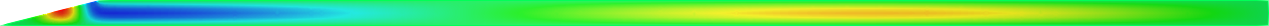}\\[-1pt]
\includegraphics[width=15cm]{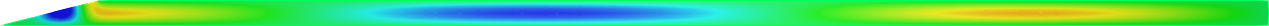}\\[-1pt]
\includegraphics[width=15cm]{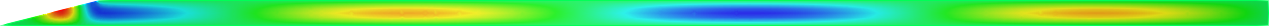}
\caption{Eigenfunctions associated with the 5 first eigenvalues of $A^\eps$ for $\eps=0.02$ and $\alpha=1.321\approx\alpha^\star_1$.\label{FigEigenfunctionsAlphaS1}}
\end{figure}

\section{Model problems around the critical angles }\label{SectionModelPb}

We arrive at the main novelty of this article. For the eigenpairs of $A^\eps$ corresponding to eigenvalues above the normalized threshold $\pi^2/\eps^2$, we obtained 1D models with at the origin either Dirichlet BC for generic $\alpha$ or Neumann BC for critical angles. The goal of this section is to describe the transition between these two models for $\alpha$ varying around critical angles. Not surprisingly, our new 1D problem will involve a Robin BC at the origin. But more interestingly, it will allow us to explain how the first eigenvalue above the normalized threshold dives below $\pi^2/\eps^2$ when $\alpha$ increases around a critical angle. Additionally, we will prove that the speed of the eigenvalue fall depends on the critical angle: it is faster around $\alpha_k^\star$, $k\ge1$, than around $\alpha_0^\star=0$. \\
\newline
To take into account the variation of the angle of the tip, we introduce a parameter $\tau\in\R$ and for $k\in\N$, we define the trapezoid
\[
\mathrm{T}^\eps_\tau\coloneqq\mathrm{T}^\eps(\alpha^\star_k+\tau\eps)= \{ z\in \R^2\,|\,y\in(0;\eps)\mbox{ and }x\in (y\tan(\alpha^\star_k+\tau\eps);1)\}
\]
(the geometry defined in (\ref{DefTrapezoid}) with $\alpha=\alpha^\star_k+\tau\eps$). We emphasize that in $\mathrm{T}^\eps_\tau$ both the small side and the angle of the tip depend on $\eps$. We denote with a subscript $\tau$ all the quantities introduced above. We wish to obtain an asymptotic expansion of the eigenpairs of $A^\eps_\tau$ as $\eps$ tends to zero with terms having a rather explicit dependence with respect to $\tau$. 

\subsection{Study around $\alpha^\star_k$, $k\ge1$}

As in Section \ref{SectionAsymptoAlpha}, we work with the method of matched asymptotic expansions. For an eigenpair $(\lambda^\eps_\tau,u^\eps_\tau)$ of $A^\eps_\tau$ with $\lambda^\eps_\tau$ close to the normalized threshold, we consider the ans\"atze
\begin{equation}\label{ExpansionTR0_2eta}
\lambda^\eps_\tau=\eps^{-2}\pi^2+\eta_\tau+\dots
\end{equation}
with 
\begin{equation}\label{FarFieldExpansion}
u^\eps_\tau(z)=\gamma_\tau(x)\sin(\pi y/\eps)+\dots\hspace{0.8cm}\qquad\mbox{(far field expansion)\hspace{0.35cm}}
\end{equation}
\begin{equation}\label{InnerFieldExpansion}
u^\eps_\tau(z)=U^0_{\tau}(z/\eps)+\eps\,U'_{\tau}(z/\eps)+\dots\qquad\mbox{(near field expansion)},
\end{equation}
where the constants $\eta_\tau$ and the functions $\gamma_\tau$, $U^0_\tau$, $U'_\tau$ are to be determined.\\
\newline
Inserting (\ref{ExpansionTR0_2eta})--(\ref{FarFieldExpansion}) into Problem (\ref{MainProblem}), we obtain 
\[
\begin{array}{|rcl}
\partial^2_x\gamma_\tau+\eta_\tau\gamma_\tau&=&0\quad\mbox{ in }I\\[3pt]
\gamma_\tau(1)&=&0.
\end{array}
\]
On the other hand, inserting (\ref{ExpansionTR0_2eta}), (\ref{InnerFieldExpansion}) into (\ref{MainProblem}), taking the limit $\eps\to0$ and collecting the terms at order $\eps^{-2}$ and $\eps^{-1}$, we find that $U^0_{\tau}$, $U'_{\tau}$ must satisfy respectively the problems
\begin{equation}\label{NearField1}
\begin{array}{|rcll}
\Delta U^0_\tau+\pi^2U^0_\tau&=&0&\mbox{ in }\Om(\alpha^\star_k)=\{(X,Y)\in \R^2\,|\,Y\in(0;1)\mbox{ and }X\in (Y\tan\alpha^\star_k;+\infty)\}\\[3pt]
U^0_\tau&=&0&\mbox{ on }\Sigma(\alpha^\star_k)=\{Z\in\partial\Om(\alpha^\star_k)\,|\,Y=0\mbox{ or }Y=1\}
\end{array}
\end{equation}
\begin{equation}\label{NearFieldP}
\begin{array}{|rcll}
\Delta U'_\tau+\pi^2U'_\tau&=&0&\mbox{ in }\Om(\alpha^\star_k)\\[3pt]
U'_\tau&=&0&\mbox{ on }\Sigma(\alpha^\star_k).
\end{array}\hspace{8.9cm}
\end{equation}
Let us clarify the conditions we should impose for $U^0_\tau$, $U'_\tau$ on the non horizontal part of the boundary $\Gamma(\alpha^\star_k)$. Introduce the coordinates 
$(s^\eps,n^\eps)$, $(s,n)$ as pictured in Figure \ref{PerturbedGeomInnerField} such that 
\[
\begin{array}{rcl}
(s^\eps,n^\eps)&=& (X\sin(\alpha^\star_k+\tau\eps)+Y\cos(\alpha^\star_k+\tau\eps),-X\cos(\alpha^\star_k+\tau\eps)+Y\sin(\alpha^\star_k+\tau\eps))\\[5pt]
(s,n)&=&(X\sin\alpha^\star_k+Y\cos\alpha^\star_k,-X\cos\alpha^\star_k+Y\sin\alpha^\star_k)\\[3pt]
&=& (\cos(\tau\eps)s^\eps+\sin(\tau\eps)n^\eps,-\sin(\tau\eps)s^\eps+\cos(\tau\eps)n^\eps).
\end{array}
\]

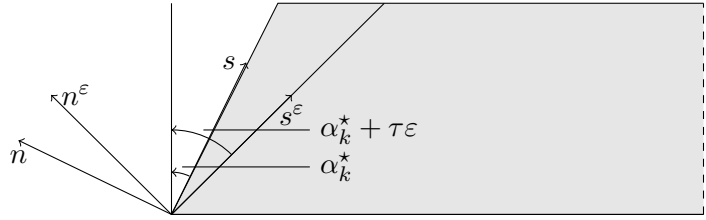
\begin{figure}[h!]
\centering
\begin{tikzpicture}[scale=1.4]
\filldraw[fill=gray!20,draw=none] (0,0) -- ++(5,0) -- ++(0,2) -- ++(-4,0)--cycle;
\draw (5,0) -- ++(-5,0) -- ++(1,2) -- ++(4,0);
\draw[dashed] (5,0)-- ++(0,2);
\draw[-] (5,0)--(0,0) -- (0,2);
\draw[-] (0,0) -- (2,2);
\draw[-] (0.1,0.45) --++ (1.2,0) node[anchor=west]{$\alpha^\star_k$};
\draw[-] (0.3,0.8) --++ (1,0) node[anchor=west]{$\alpha^\star_k+\tau\eps$};
\draw[<-] (0,0) ++(90:.4) arc (90:65:.4); 
\draw[<-] (0,0) ++(90:0.8) arc (90:45:0.8); 
\begin{scope} [rotate=64]
\draw[->] (0,0) -- (1.6,0) node[anchor=east]{$s$};
\draw[->] (0,0) -- (0,1.6) node[anchor=north]{$n$};
\end{scope}
\begin{scope} [rotate=45]
\draw[->] (0,0) -- (1.6,0) node[anchor=north]{$s^\eps$};
\draw[->] (0,0) -- (0,1.6) node[anchor=west]{$n^\eps$};
\end{scope}
\end{tikzpicture}
\caption{Near field geometries $\Om(\alpha^\star_k)$ and $\Om(\alpha^\star_k+\tau\eps)$, $\tau>0$.}\label{PerturbedGeomInnerField}
\end{figure}

\noindent We have
\[
\cfrac{\partial}{\partial n^\eps}=\sin(\tau\eps)\cfrac{\partial}{\partial s}+\cos(\tau\eps)\cfrac{\partial}{\partial n}
\]
and on the line such that $n^\eps=0$, there holds $n=-\tan(\tau\eps)s$. Therefore, the condition 
$\partial_{n^\eps}u^\eps_\tau=0$ on $\Gamma^\eps_\tau$ gives
\[
\begin{array}{rcl}
0&=&\partial_nU^0_\tau(s,-\tan(\eps\tau)s)+\eps\,(\tau\partial_sU^0_\tau(s,-\tan(\eps\tau)s)+\partial_nU'_\tau(s,-\tan(\eps\tau)s))+O(\eps^2)\\[3pt]
&=&\partial_nU^0_\tau(s,0)+\eps\,(\tau\partial_sU^0_\tau(s,0)+\partial_nU'_\tau(s,0)-\tau s\partial^2_nU^0_\tau(s,0))+O(\eps^2).
\end{array}
\]
This leads us to impose 
\begin{equation}\label{BdryCondition0}
\partial_nU^0_\tau=0\qquad\mbox{ on }\Gamma(\alpha^\star_k)
\end{equation}
and, using the first line of (\ref{NearField1}), 
\begin{equation}\label{BdryConditionP}
\partial_nU'_\tau=-\tau\partial_sU^0_\tau-\tau s(\partial^2_sU^0_\tau+\pi^2U^0_\tau)\qquad\mbox{ on }\Gamma(\alpha^\star_k).
\end{equation}
To define completely the terms in (\ref{ExpansionTR0_2eta}), 
(\ref{FarFieldExpansion}), (\ref{InnerFieldExpansion}), it remains to impose a condition at the origin for $\gamma_\tau$ and to prescribe behaviors at infinity for $U^0_\tau$, $U'_\tau$. Again, this will be done by matching the far field and near field expansions (see (\ref{FarFieldExpansion}) and (\ref{InnerFieldExpansion})) of $u^\eps_\tau$ in some intermediate region where $x\to0$ and $x/\eps\to+\infty$ as $\eps\to0$.\\ 
\newline
For the far field expansion, as in (\ref{TaylorFarField0}), at the origin we have the Taylor series
\begin{equation}\label{TaylorFarField}
\begin{array}{rcl}
\gamma_\tau(x)\sin(\pi y/\eps) &= &(\gamma_\tau(0)+x\partial_x\gamma_\tau(0)+\dots)\sin(\pi y/\eps) \\[4pt]
&= &(\gamma_\tau(0)+\eps (x/\eps)\partial_x\gamma_\tau(0)+\dots)\sin(\pi y/\eps).
\end{array}
\end{equation}
From (\ref{NearField1}), (\ref{BdryCondition0}) and (\ref{TaylorFarField}), when matching the expansions at order $\eps^0$, it is natural to set 
\begin{equation}\label{IdentiTerm1}
U^0_{\tau}=\gamma_\tau(0)\,W
\end{equation}
where $W$ is the function defined in (\ref{ScaSol}). Here, because $\alpha^\star_k$ is a critical angle, $W$ is real and admits the decomposition 
\[
W(Z)=\sin(\pi Y)+\widetilde{W}(Z),
\]
with $\widetilde{W}\in\mH^1_0(\Om;\Sigma)$ which decays exponentially at infinity.\\
\newline
The problem (\ref{NearFieldP}), (\ref{BdryConditionP}) in general has no bounded solution. However the classical Kondratiev theory \cite{Kond67} (see also \cite[Chap.\,2 and 5]{NaPl94}) guarantees that it admits solutions with linear growth of the form
\begin{equation}\label{ExpanTermAffine}
U'_{\tau}(Z)=(C_0+C_1X)\sin(\pi Y)+\tilde{U}'_{\tau}(Z)
\end{equation}
where $C_0$, $C_1\in\R$ and $\tilde{U}'_{\tau}\in\mH^1_0(\Om(\alpha_k^\star);\Sigma(\alpha_k^\star))$. Note that the constant $C_0$ is arbitrary, because $U'_{\tau}$ is defined up to $\mrm{span}(W)$, whereas $C_1$ is uniquely defined. Let us compute $C_1$. To proceed, first using (\ref{NearField1}), (\ref{NearFieldP}), (\ref{BdryCondition0}), we write, for $L$ large enough,
\begin{equation}\label{Calculus1}
\int_{\Gamma(\alpha^\star_k)} W\partial_n U'_{\tau}\,ds=-\int_{\Gamma_L} W\partial_X U'_{\tau}-U'_{\tau}\partial_X W\,dY,
\end{equation}
where we recall that $\Gamma_L=\{L\}\times(0;1)$. Taking the limit $L\to+\infty$ in (\ref{Calculus1}) and exploiting that $\partial_X W$ decays exponentially as $X\to+\infty$, we obtain 
\[
\int_{\Gamma(\alpha^\star_k)} W\partial_n U'_{\tau}\,ds=-\cfrac{C_1}{2}\,.
\]
On the other hand, from (\ref{BdryConditionP}), (\ref{IdentiTerm1}), we can write
\[
\int_{\Gamma(\alpha^\star_k)} W\partial_n U'_{\tau}\,ds=-\tau\gamma_\tau(0)\int_{\Gamma(\alpha^\star_k)} W(\partial_sW+ s(\partial^2_sW+\pi^2W))\,ds = \tau B\gamma_\tau(0)
\]
with 
\begin{equation}\label{CalculusB}
B\coloneqq\int_{\Gamma(\alpha^\star_k)}s\big((\partial_sW)^2-\pi^2W^2\big)\,ds.
\end{equation}
In Proposition \ref{PropositionSigneIntegral} below, we will show that $B>0$ (and explain why the integral defining $B$ is finite). Thus in (\ref{ExpanTermAffine}) we find $C_1=-2\tau B \gamma_\tau(0)$.\\
As a consequence, by matching the linear terms of order $\eps$ between (\ref{TaylorFarField}) and (\ref{InnerFieldExpansion}), we see that we must impose the Robin BC at the origin
\[
\partial_x\gamma_\tau(0)=-2\tau B\gamma_\tau(0).
\]
Finally for the far field term in (\ref{ExpansionTR0_2eta}), we obtain the spectral model problem 
\begin{equation}\label{farfieldetaeta2}
\begin{array}{|rcl}
\partial^2_x\gamma_\tau+\eta_\tau\gamma_\tau&=&0\quad\mbox{ in }I\\[3pt]
\gamma_\tau(1)&=&0\\[3pt]
\partial_x\gamma_\tau(0)&=&-2\tau B\gamma_\tau(0).
\end{array}
\end{equation}

\noindent Now we can state and prove the following result. As in Theorem \ref{MainThm1}, $N_\circ+N_\dagger$ denotes the number of eigenvalues of $A^{\Om(\alpha^\star_k)}$ less or equal than $\pi^2$.

\begin{theorem}\label{ThmModel1}
Consider a critical angle $\alpha_k^\star$, $k\ge 1$, introduced in (\ref{DefAlphaP}) and some $\tau\in\R$.\\ Let $\lambda^\eps_{p,\tau}$ denote the $p$-th eigenvalue of the operator $A^\eps_\tau$ defined in the geometry $\mathrm{T}^\eps_\tau=\mathrm{T}^\eps(\alpha^\star_k+\tau\eps)$.\\ For $p=N_\circ+N_\dagger+q$, $q\in\N^{\ast}$, there are  constants $C_{p,\tau}$, $\delta_{p,\tau}$, $\eps_{p,\tau}>0$ such that we have the estimate 
\[
\big|\lambda^\eps_{p,\tau}-(\eps^{-2}\pi^2+\eta_{q,\tau})\big|\le C_p\,\eps^{\delta_p}\qquad\quad\forall\eps\in(0;\eps_p],
\]
where $\eta_{q,\tau}$ stands for the $q$-th eigenvalue of the model problem (\ref{farfieldetaeta2}) with Robin BC at the origin.\\
[5pt]
Moreover, we have
\begin{equation}\label{FormulaModel1a}
\lim_{\tau\to-\infty}\eta_{1,\tau}=\pi^2,\qquad\qquad\quad \eta_{1,\tau}\underset{\tau\to+\infty}{\sim}-4B^2\tau^2\hspace{0.5cm}
\end{equation}
and for $q\ge2$,
\begin{equation}\label{FormulaModel1b}
\lim_{\tau\to-\infty}\eta_{q,\tau}=q^2\pi^2,\qquad\qquad \lim_{\tau\to+\infty}\eta_{q,\tau}=(q-1)^2\pi^2.\hspace{-0.3cm}
\end{equation}
Here the constant $B>0$ is defined in (\ref{CalculusB}) and $N_\circ$, $N_\dagger$ are introduced respectively in (\ref{DefEigenValDS}), (\ref{DefNbEigTR}).
\end{theorem}
\begin{remark}
The constants $C_{p,\tau}$, $\delta_{p,\tau}$, $\eps_{p,\tau}>0$ depend \textit{a priori} on $\tau$. However they can be chosen uniform for $\tau$ in a compact set.
\end{remark}
\begin{proof}
Let us complement the above analysis by detailing the study of (\ref{farfieldetaeta2}) leading to formula (\ref{FormulaModel1a}), (\ref{FormulaModel1b}). First, observe that for $\tau=0$, we find as in (\ref{farfield_BC_N}) a homogeneous Neumann condition at the origin. Thus for $q\in\N^\ast$, we have
\[
\eta_{q,0}=(q-1/2)^2\pi^2,\qquad\qquad  \gamma_{q,0}(x)=C \cos(\pi(q-1/2)x),
\]
where $C$ is a normalization factor. The variational formulation associated with (\ref{farfieldetaeta2}) writes
\begin{equation}\label{VariationalForm1}
\begin{array}{|l}
\mbox{ Find }(\eta,\gamma)\in\R\times \mH^1_{0}(I;1)\mbox{ such that }\\[3pt]
\dsp\int_I \partial_x\gamma\,\partial_x\gamma' \,dx-2\tau B\gamma(0)\gamma'(0) = \eta \int_I \gamma\gamma' \,dx\qquad \forall \gamma'\in\mH^1_{0}(I;1),
\end{array}
\end{equation}
where $\mH^1_{0}(I;1)\coloneqq\{\gamma\in\mH^1(I)\,|\,\gamma(1)=0\}$. With (\ref{VariationalForm1}) it is easy to show that the eigenvalues of (\ref{farfieldetaeta2}) are decreasing with respect to $\tau\in\R$.\\ 
Solving (\ref{farfieldetaeta2}) explicitly, we find that $\eta_\tau<0$ is an eigenvalue if and only if there holds 
\begin{equation}\label{Disper1}
\coth(\sqrt{|\eta_\tau|})=\cfrac{2\tau B}{\sqrt{|\eta_\tau|}}\,,
\end{equation}
where $\coth$ stands for the hyperbolic cotangent. By remarking that $x\mapsto x \coth(x)$ is larger than $1$ on $(0;+\infty)$, since $B>0$, we deduce that (\ref{Disper1}) has exactly one solution $\eta_{\tau,1}$ if and only if $\tau>1/(2B)$. Moreover, we find $\eta_{\tau,1}\sim-4B^2\tau^2$ when $\tau\to+\infty$.\\
Besides, $0$ is an eigenvalue of (\ref{farfieldetaeta2}) if and only if $\tau=1/(2B)$.\\ Finally we obtain that $\eta_\tau>0$ is an eigenvalue of (\ref{farfieldetaeta2}) if and only if there holds 
\begin{equation}\label{Disper2}
\cot(\sqrt{\eta_\tau})=\cfrac{2\tau B}{\sqrt{\eta_\tau}}\,.
\end{equation}
Therefore we have $\lim_{\tau\to-\infty}\eta_{q,\tau}=q^2\pi^2$ and $\lim_{\tau\to+\infty}\eta_{q,\tau}=(q-1)^2\pi^2$. The presentation above is rather formal but error estimates can be obtained by working as explained in the end of the proof of Theorem \ref{MainThm1}. 
\end{proof}
\noindent To illustrate the results of Theorem \ref{ThmModel1}, we solve the 1D model problem (\ref{VariationalForm1}) with a finite element method and display in Figure \ref{BehaviourEigen} the five first eigenvalues for $\tau$ varying in $(-30;30)$. As $\tau\to+\infty$, as expected we observe that the first eigenvalue dives below zero. On the other hand, the positive eigenvalues converge to the eigenvalues $q^2\pi^2$, $q\in\N^\ast$, of the Dirichlet problem (\ref{farfield}), (\ref{farfield_BC}) as $\tau\to\pm\infty$.

\begin{figure}[ht!]
\centering
\includegraphics[width=11cm]{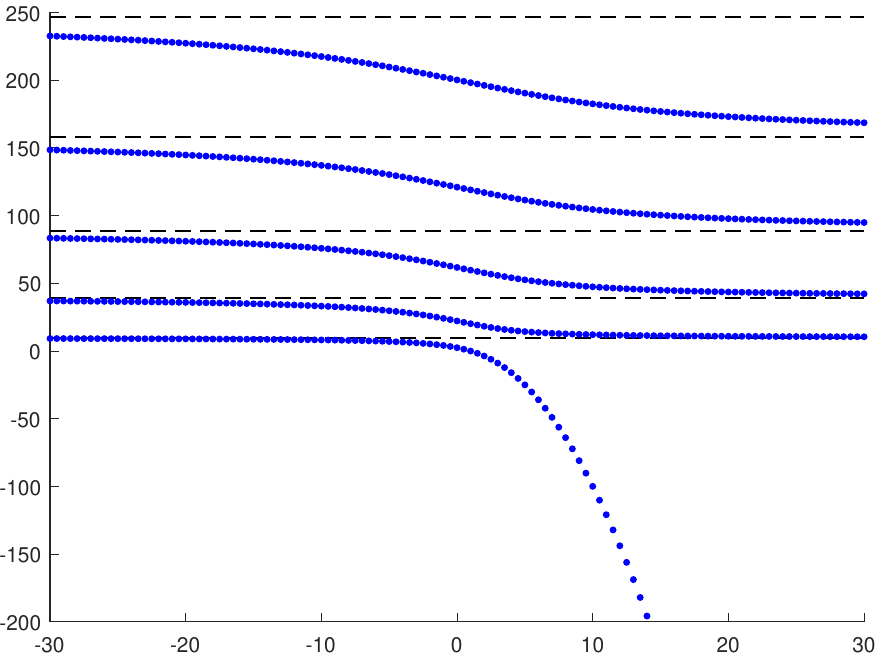}
\caption{Five first eigenvalues of the 1D model problem (\ref{VariationalForm1}) for $\tau$ varying in $(-30;30)$. The horizontal dashed lines correspond to the values of $q^2\pi^2$, $q=1,2,\dots,5$. Here we take $2B=1$.\label{BehaviourEigen}}
\end{figure}

\subsection{Study around $\alpha^\star_0=0$}

Now we turn our attention to the perturbation of the angle around $\alpha^\star_0=0$ (straight end). Let us first explain why the analysis above must be adapted.\\
\newline
Pick some $\tau\in\R$. Let us try to derive the asymptotics of the eigenpairs of the operator $A^\eps_\tau$ defined in the geometry $\mathrm{T}^\eps_\tau=\mathrm{T}^\eps(\tau\eps)$ with the above procedure. In $\Om(\alpha^\star_0)=\Om(0)=(0;+\infty)\times(0;1)$, the function $W$ defined in (\ref{ScaSol}) is simply $W(Z)=\sin(\pi Y)$. As a consequence (\ref{IdentiTerm1}) leads to $
U^0_{\tau}=\gamma_\tau(0)\,\sin(\pi Y)$ so that from (\ref{NearFieldP}), (\ref{BdryConditionP}), we find that $U'_\tau$  solves
\[
\begin{array}{|rcll}
\Delta U'_\tau+\pi^2U'_\tau&=&0&\quad\mbox{ in }\Om(0)\\[3pt]
U'_\tau&=&0&\quad\mbox{ on }\Sigma(0)=\{Z\in\partial\Om(0)\,|\,Y=0\mbox{ or }1\}\\[3pt]
\partial_nU'_\tau&=&-\pi\gamma_\tau(0)\,\cos(\pi Y)&\quad\mbox{ on }\Gamma(0)=\{0\}\times(0;1).
\end{array}
\]
Since $Z\mapsto \cos(\pi Y)$ is odd with respect to the line $Y=1/2$, using decomposition in Fourier series with respect to the $Y$-variable, one shows that $U'_\tau$ is exponentially decaying at infinity. As a consequence, we cannot match the linear terms of order $\eps$ between (\ref{TaylorFarField}) and (\ref{InnerFieldExpansion}). This indicates that the ans\"atze (\ref{ExpansionTR0_2eta})--(\ref{InnerFieldExpansion}) is not well-suited, we have to modify it.\\
\newline
Instead, let us work in the domain 
\[
\mathrm{T}^\eps_\tau(\kappa)=\{ z\in\R^2\,|\,y\in(0;\eps)\mbox{ and }x\in (y\tan\kappa;1)\},
\]
where for the moment we do not specify the dependence of $\kappa$ (small) with respect to $\tau,\eps$. Let us still denote by $A^\eps_\tau$ the operator associated with Problem (\ref{MainProblem}) in the geometry $\mathrm{T}^\eps_\tau(\kappa)$. As an expansion of an eigenpair $(\lambda^\eps_\tau,u^\eps_\tau)$ of $A^\eps_\tau$ with $\lambda^\eps_\tau$ close to the normalized threshold, keep (\ref{ExpansionTR0_2eta})--(\ref{FarFieldExpansion}) and replace (\ref{InnerFieldExpansion}) by the new decomposition
\begin{equation}\label{InnerFieldExpansionNew}
u^\eps_\tau(z)=U^0_{\tau}(z/\eps)+\kappa\,U'_{\tau}(z/\eps)+\kappa^2\,U''_{\tau}(z/\eps)+\dots \qquad\mbox{(near field expansion).}
\end{equation}
Here 
\begin{equation}\label{IdentiTerm1Bis}
U^0_{\tau}(z/\eps)=\gamma_\tau(0)\,W(z/\eps)=\gamma_\tau(0)\,\sin(\pi y/\eps)
\end{equation}
as in (\ref{IdentiTerm1}) and $U'_{\tau}$, $U''_{\tau}$ are to be determined.\\
\newline
Inserting (\ref{ExpansionTR0_2eta}), (\ref{InnerFieldExpansionNew}) into (\ref{MainProblem}), taking the limit $\eps\to0$ and collecting the terms at order $\kappa\eps^{-2}$ and $\kappa^2\eps^{-2}$, we find that $U'_{\tau}$, $U''_{\tau}$ must satisfy respectively the problems
\begin{equation}\label{NearFieldPNew}
\begin{array}{|rcll}
\Delta U'_\tau+\pi^2U'_\tau&=&0&\quad\mbox{ in }\Om(0)\\[3pt]
U'_\tau&=&0&\quad\mbox{ on }\Sigma(0)
\end{array}
\end{equation}
\begin{equation}\label{NearFieldPP}
\begin{array}{|rcll}
\Delta U''_\tau+\pi^2U''_\tau&=&0&\quad\mbox{ in }\Om(0)\\[3pt]
U'_\tau&=&0&\quad\mbox{ on }\Sigma(0).
\end{array}\hspace{-0.15cm}
\end{equation}
Let us clarify the conditions we should impose on $\Gamma(0)=\{0\}\times(0;1)$ for $U'_\tau$, $U''_\tau$. Introduce the coordinates $(s^\kappa,n^\kappa)$, $(s,n)$ such that 
\[
\begin{array}{rcl}
(s^\kappa,n^\kappa)&=& (X\sin\kappa+Y\cos\kappa,-X\cos\kappa+Y\sin\kappa)\\[5pt]
(s,n)&=&(Y,-X).
\end{array}
\]
We have
\[
\cfrac{\partial}{\partial n^\kappa}=\sin\kappa\cfrac{\partial}{\partial s}+\cos\kappa\cfrac{\partial}{\partial n}
\]
and on the line such that $n^\kappa=0$, there holds $n=-s\tan\kappa $. Therefore, the condition 
$\partial_{n^\kappa}u^\eps_\tau=0$ on $\Gamma^\eps_\tau$ implies
\[
\begin{array}{l}
0=\partial_nU^0_\tau(s,-s\tan\kappa)+\kappa\,(\partial_sU^0_\tau(s,-s\tan\kappa)+\partial_nU'_\tau(s,-s\tan\kappa))\\[4pt]
\hspace{3.7cm}+\kappa^2(\partial_sU'_\tau(s,-s\tan\kappa)+\partial_nU''_\tau(s,-s\tan\kappa)-\frac{1}{2}\partial_nU^0_\tau(s,-s\tan\kappa))+O(\kappa^3).
\end{array}
\]
Using that $\partial_nU^0_\tau=0$ according to (\ref{IdentiTerm1Bis}), this gives
\[
\begin{array}{rcl}
0&=&\kappa\,(\partial_sU^0_\tau(s,0)+\partial_nU'_\tau(s,0))+\kappa^2\,(\partial_sU'_\tau(s,0)+\partial_nU''_\tau(s,0)- s\partial^2_nU'_\tau(s,0))+O(\kappa^3).
\end{array}
\]
Thus we impose 
\begin{equation}\label{BdryConditionPNew}
\partial_nU'_\tau=-\pi\gamma_\tau(0)\,\cos(\pi Y)\qquad\mbox{ on }\Gamma(0)
\end{equation}
and
\begin{equation}\label{BdryConditionPP}
\partial_nU''_\tau=-\partial_sU'_\tau- s(\partial^2_sU'_\tau+\pi^2U'_\tau)\qquad\mbox{ on }\Gamma(0).
\end{equation}
The solution $U'_\tau$ of (\ref{NearFieldPNew}), (\ref{BdryConditionPNew}) can be decomposed as $U'_\tau=\gamma_\tau(0) V$ where $V$ solves
\[
\begin{array}{|rcll}
\Delta V+\pi^2V&=&0&\quad\mbox{ in }\Om(0)\\[3pt]
V&=&0&\quad\mbox{ on }\Sigma(0)\\[3pt]
\partial_nV&=&-\pi\,\cos(\pi Y)&\quad\mbox{ on }\Gamma(0).
\end{array}
\] 
As already mentioned, $U'_\tau$ is exponentially decaying at infinity. On the other hand, again from \cite{Kond67}, \cite[Chap.\,5]{NaPl94}, we know that (\ref{NearFieldPP}), (\ref{BdryConditionPP}) admits a solution with the decomposition 
\begin{equation}\label{ExpanTermAffineNew}
U''_{\tau}(Z)=(C_0+C_1X)\sin(\pi Y)+\tilde{U}''_{\tau}(Z)
\end{equation}
where $C_0$, $C_1\in\R$ and $\tilde{U}''_{\tau}\in\mH^1_0(\Om(0);\Sigma(0))$. Note that the constant $C_0$ is arbitrary whereas $C_1$ is uniquely defined. Let us compute its value. To proceed, first using (\ref{NearFieldP}), (\ref{NearFieldPP}), (\ref{BdryConditionP}), we write, for $L>0$,
\begin{equation}\label{Calculus1New}
\int_{\Gamma(0)} W\partial_n U''_{\tau}\,ds=-\int_{\Gamma_L} W\partial_X U''_{\tau}\,dY,
\end{equation}
where again $\Gamma_L=\{L\}\times(0;1)$. Taking the limit $L\to+\infty$ in (\ref{Calculus1New}), we obtain 
\[
\int_{\Gamma(0)} W\partial_n U''_{\tau}\,ds=-\cfrac{C_1}{2}\,.
\]
On the other hand, from (\ref{BdryConditionPP}), (\ref{IdentiTerm1}) as well as (\ref{BdryConditionPNew}), we obtain
\[
\begin{array}{rcl}
\dsp\int_{\Gamma(0)} W\partial_n U''_{\tau}\,ds&=&-\dsp\int_{\Gamma(0)} \sin(\pi Y)(\partial_sU'_{\tau}+ s(\partial^2_sU'_{\tau}+\pi^2U'_{\tau}))\,ds \\[8pt]
&=&-\dsp\int_{\Gamma(0)} \sin(\pi Y)(\partial_sU'_{\tau}-\partial_sU'_{\tau}+\pi^2sU'_{\tau})-\pi\cos(\pi Y)s\partial_sU'_{\tau}\,ds \\[8pt]
&=&-\dsp\int_{\Gamma(0)} \pi^2\sin(\pi Y)sU'_{\tau}-\pi\cos(\pi Y)s\partial_sU'_{\tau}\,ds \\[8pt]
&=&-\dsp\int_{\Gamma(0)} \pi\cos(\pi Y)U'_{\tau}\,ds=\gamma_\tau(0)\int_{\Gamma(0)} V\partial_nV\,ds =\gamma_\tau(0)\,D
\end{array}
\]
with 
\begin{equation}\label{CalculusD}
D\coloneqq\int_{\Om(0)}|\nabla V|^2-\pi^2V^2\,dZ.
\end{equation}
Using that $V$ is odd with respect to the line $Y=1/2$, with the Poincar\'e-Friedrichs inequality, one proves that
\[
D\ge 4\pi^2\int_{\Om(0)}V^2\,dZ-\pi^2\int_{\Om(0)}V^2\,dZ=3\pi^2\int_{\Om(0)}V^2\,dZ,
\]
which guarantees that $D>0$. Thus in (\ref{ExpanTermAffineNew}) we obtain $C_1=-2D\gamma_\tau(0)$. As a consequence, by matching (\ref{TaylorFarField}) and (\ref{InnerFieldExpansionNew}), we see that we must impose $\eps\partial_x\gamma_\tau(0)=-2D\kappa^2\gamma_\tau(0)$. This leads us to set 
\[
\kappa=\tau\sqrt{\eps}\qquad\mbox{ with }\qquad\tau\in\R\qquad\mbox{ and }\qquad\eps>0. 
\]
Then we have obtained the Robin BC at the origin
\[
\partial_x\gamma_\tau(0)=-2\tau^2D\gamma_\tau(0).
\]
Thus finally for the far field term in (\ref{ExpansionTR0_2eta}), we deduce the spectral problem 
\begin{equation}\label{PbRobin2}
\begin{array}{|rcl}
\partial^2_x\gamma_\tau+\eta_\tau\gamma_\tau&=&0\quad\mbox{ in }I\\[3pt]
\gamma_\tau(1)&=&0\\[3pt]
\partial_x\gamma_\tau(0)&=&-2\tau^2D\gamma_\tau(0).
\end{array}
\end{equation}
\noindent From this analysis, we can state and demonstrate the following result:
\begin{theorem}\label{ThmModel2}
Pick some $\tau\in\R$. Let $\lambda^\eps_{p,\tau}$ denote the $p$-th eigenvalue of the operator $A^\eps_\tau$ defined in the geometry $\mathrm{T}^\eps_\tau=\mathrm{T}^\eps(\tau\sqrt{\eps})$ (small perturbation of the straight end). For $p\in\N^{\ast}$, there are  constants $C_{p,\tau}$, $\delta_{p,\tau}$, $\eps_{p,\tau}>0$ such that we have the estimate 
\[
\big|\lambda^\eps_{p,\tau}-(\eps^{-2}\pi^2+\eta_{p,\tau})\big|\le C_p\,\eps^{\delta_p}\qquad\quad\forall\eps\in(0;\eps_p],
\]
where $\eta_{p,\tau}$ stands for the $p$-th eigenvalue of the model problem (\ref{PbRobin2}) with Robin BC at the origin.\\
[5pt]
Moreover, we have
\begin{equation}\label{FormulaModel2a}
\eta_{1,\tau}\underset{\tau\to\pm\infty}{\sim}-4D^2\tau^4
\end{equation}
and for $p\ge2$,
\begin{equation}\label{FormulaModel2b}
\lim_{\tau\to\pm\infty}\eta_{p,\tau}=(p-1)^2\pi^2.
\end{equation}
Here the constant $D>0$ is defined in (\ref{CalculusD}) and $N_\circ$, $N_\dagger$ are introduced respectively in (\ref{DefEigenValDS}), (\ref{DefNbEigTR}).
\end{theorem}
\begin{remark}
Again, the constants $C_{p,\tau}$, $\delta_{p,\tau}$, $\eps_{p,\tau}>0$ depend \textit{a priori} on $\tau$ but they can be chosen uniform for $\tau$ varying in a compact set.
\end{remark}
\begin{proof}
Let us briefly explain how to obtain (\ref{FormulaModel2a}), (\ref{FormulaModel2b}) from (\ref{CalculusD}). First for $\tau=0$ in (\ref{PbRobin2}), as in (\ref{farfield_BC_N}), we get a homogeneous Neumann condition at the origin. Therefore we obtain, for $p\in\N^\ast$,
\[
\eta_{p,0}=(p-1/2)^2\pi^2,\qquad\qquad  \gamma_{p,0}(x)=C \cos(\pi(p-1/2)x),
\]
where $C$ is a normalization factor. The variational formulation associated with (\ref{PbRobin2}) writes
\begin{equation}\label{VariationalForm1New}
\begin{array}{|l}
\mbox{ Find }(\eta,\gamma)\in\R\times \mH^1_{0}(I;1)\mbox{ such that }\\[3pt]
\dsp\int_I \partial_x\gamma\partial_x\gamma' \,dx-2\tau^2D\gamma(0)\gamma'(0) = \eta \int_I \gamma\gamma' \,dx\qquad \forall \gamma'\in\mH^1_{0}(I;1),
\end{array}
\end{equation}
We remark that the eigenvalues of (\ref{PbRobin2}) are even with respect to $\tau$. Therefore, it suffices to study their features for $\tau\ge0$. Observe that the spectrum of (\ref{PbRobin2}) for $\tau=\tau_1$ coincides with the spectrum  of (\ref{farfieldetaeta2}) for $\tau=D\tau_1^2/B$. We deduce that the eigenvalues  of (\ref{PbRobin2}) are decreasing with respect to $\tau$ in $(0;+\infty)$. Moreover, by working as in (\ref{Disper1}), (\ref{Disper2}), we find 
\[
\eta_{1,\tau}\underset{\tau\to+\infty}{\sim}-4D^2\tau^4
\]
and for $p\ge2$, $\lim_{\tau\to\pm\infty}\eta_{p,\tau}=(p-1)^2\pi^2$. Again, we adopt here a formal presentation for the sake of conciseness but rigorous error estimates can be obtained.
\end{proof}
\noindent Let us comment more on the asymptotics for the diving eigenvalue. When comparing at first glance estimates (\ref{FormulaModel1a}) and (\ref{FormulaModel2a}), it seems that the fall is faster at $\alpha_0^\star$ than at $\alpha$ above $\alpha_k^\star$, $k\ge1$. Let us explain why it is not the case. Redenote by $\tau_+$, $\eta^+_{1}$ (resp. $\tau_0$, $\eta^0_{1}$) the quantities $\tau$, $\eta_{1}$ appearing in (\ref{FormulaModel1a}) (resp. (\ref{FormulaModel2a})). To get similar variations of angle $\alpha$ around $\alpha_k^\star$ and $\alpha_0^\star$, we must impose $\tau_+\eps=\tau_0\sqrt{\eps}$, that is  $\tau_0=\tau_+\sqrt{\eps}$. Thus estimate (\ref{FormulaModel2a}) writes
\[
\eta^0_{1,\tau_+}\underset{\tau_+\to+\infty}{\sim}-4D^2\tau_+^4\eps^2.
\]
Next it should be remembered that $\tau_+$ is allowed to tend to $+\infty$ but with $\tau_+\eps$ small compared to one. If we take $\tau_+=c\,\eps^{-1+s}$ with $c,s>0$, we obtain 
\[
\eta^+_{1,\tau_+}\underset{\tau_+\to+\infty}{\sim}-4B^2c^2\eps^{-2+2s},\qquad\qquad
\eta^0_{1,\tau_+}\underset{\tau_+\to+\infty}{\sim}-4D^2c^4\eps^{-2+4s}.
\]
This shows that $\eta^+_{1,\tau_+}$ dives faster to $-\infty$ than $\eta^0_{1,\tau_+}$ as $\eps\to0$.

\begin{figure}[h!]
\centering
\includegraphics[width=10.4cm]{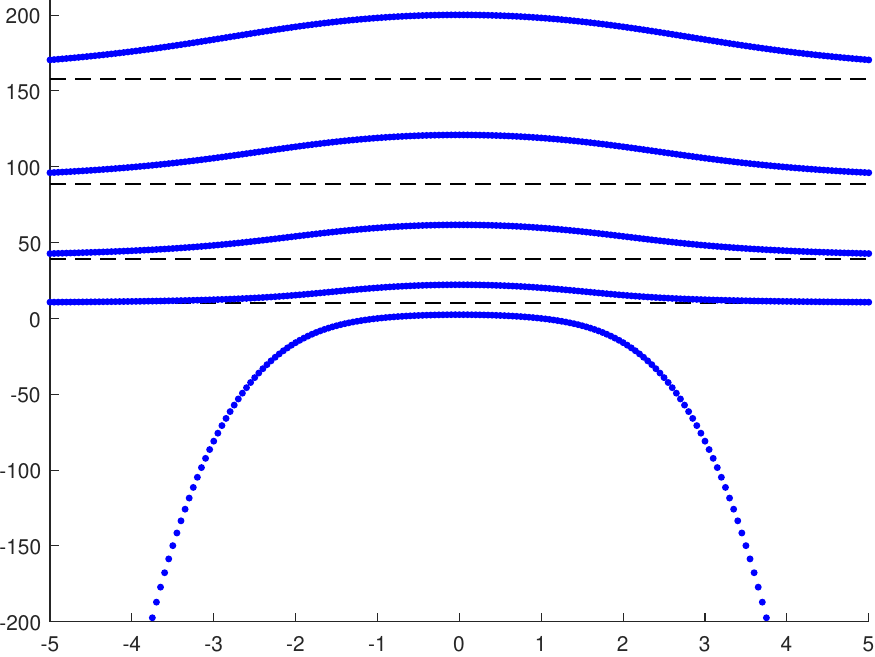}\vspace{-0.3cm}
\caption{Five first eigenvalues of the model problem (\ref{PbRobin2}) for $\tau$ varying in $(-5;5)$. The horizontal dashed lines correspond to the values of $q^2\pi^2$, $q=1,2,\dots,5$. Here we take $2D=1$.\label{BehaviourEigenZero}}
\end{figure}

\noindent To illustrate the results of Theorem \ref{ThmModel2}, we solve (\ref{VariationalForm1New}) with a finite element method and represent the five first eigenvalues for $\tau_0$ varying in $(-5;5)$ in Figure \ref{BehaviourEigenZero}. Again, let us emphasize that for $\tau_+$ sweeping the interval $(-30;30)$ for the model Problem (\ref{farfieldetaeta2}), $\tau_0$ belongs to $(-30\sqrt{\eps};30\sqrt{\eps})$ for the model Problem (\ref{PbRobin2}). For $\eps=0.02$ as in the numerics of Figure \ref{SpectrumVsAlpha}, this gives approximately $\tau_0\in(-5;5)$. In other words, the ranges of $\tau$ in Figures \ref{BehaviourEigen}, \ref{BehaviourEigenZero} have been fixed so that they correspond to the same interval of variation of $\alpha$ for models (\ref{farfieldetaeta2}) and (\ref{PbRobin2}). The results obtained in Figure \ref{BehaviourEigenZero} corroborate formulas (\ref{FormulaModel2a})-(\ref{FormulaModel2b}). We also clearly observe the milder variation of the spectrum with respect to the perturbation of the angle compared to Figure \ref{BehaviourEigen}. Let us mention that this mild behavior of the eigenvalues at the origin is not completely surprising due to 
the evenness and smoothness with respect to the perturbation of the geometry.

\subsection{Numerical experiments in the broken strip}\label{ParaNumBrokenStrip}

\begin{figure}[h!]
\centering
\includegraphics[width=5.2cm]{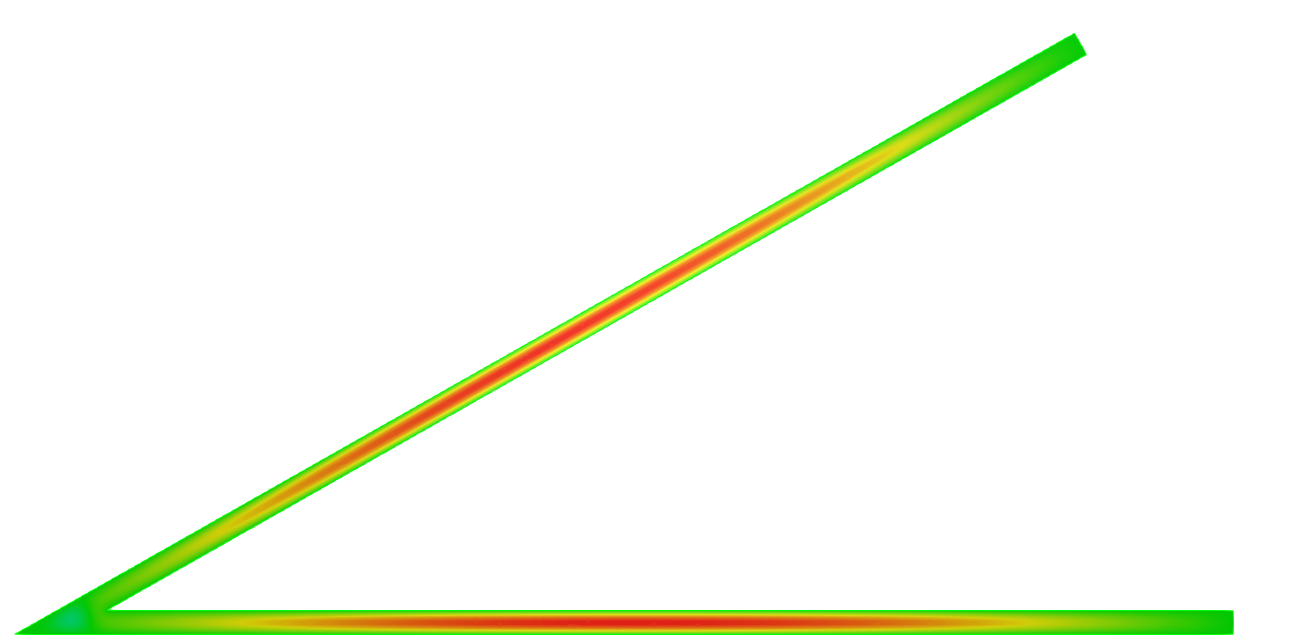}\  \includegraphics[width=5.2cm]{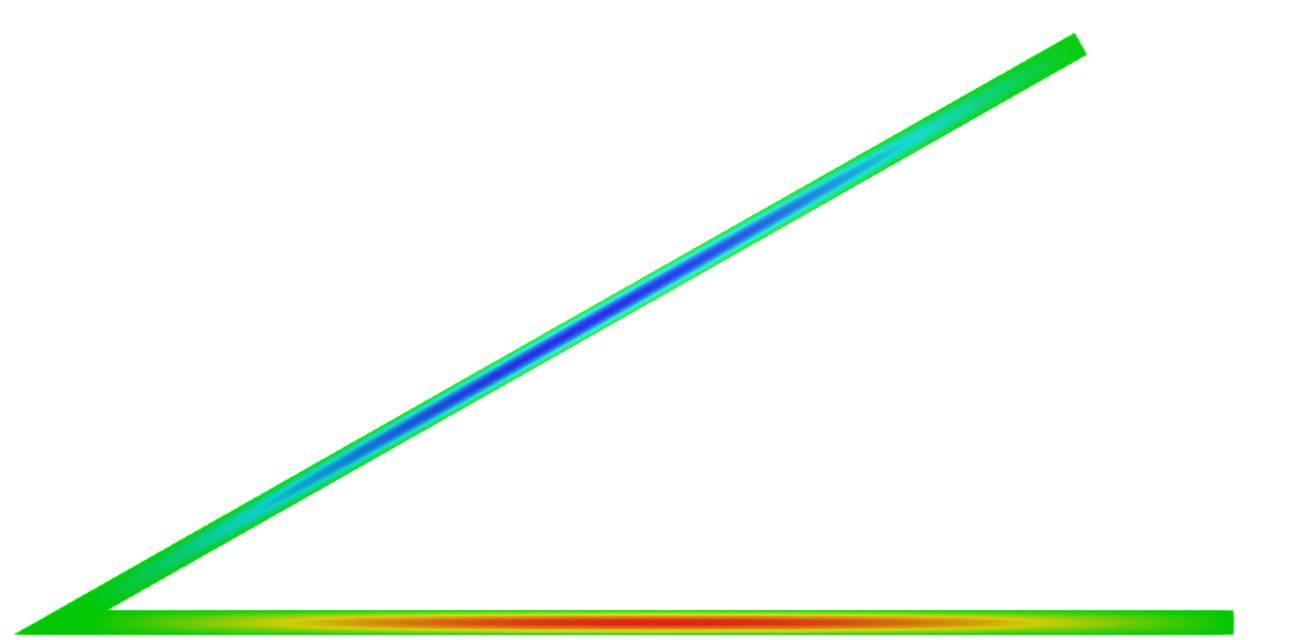}\  \includegraphics[width=5.2cm]{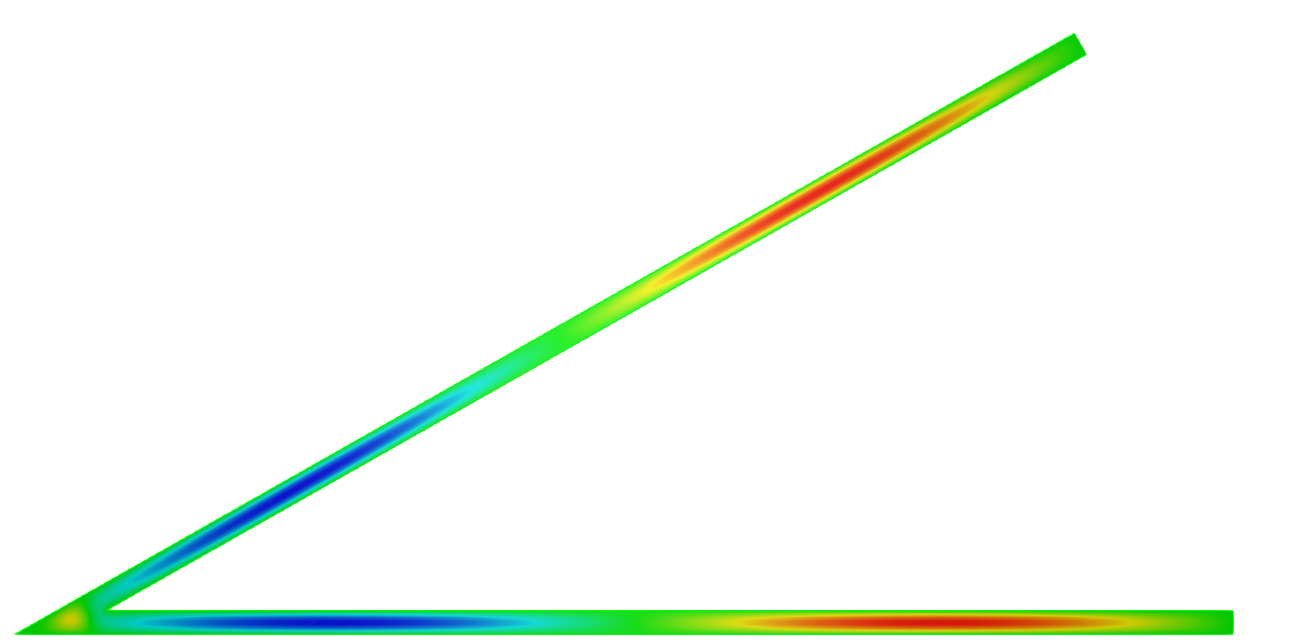}\\[3pt]
\includegraphics[width=5.2cm]{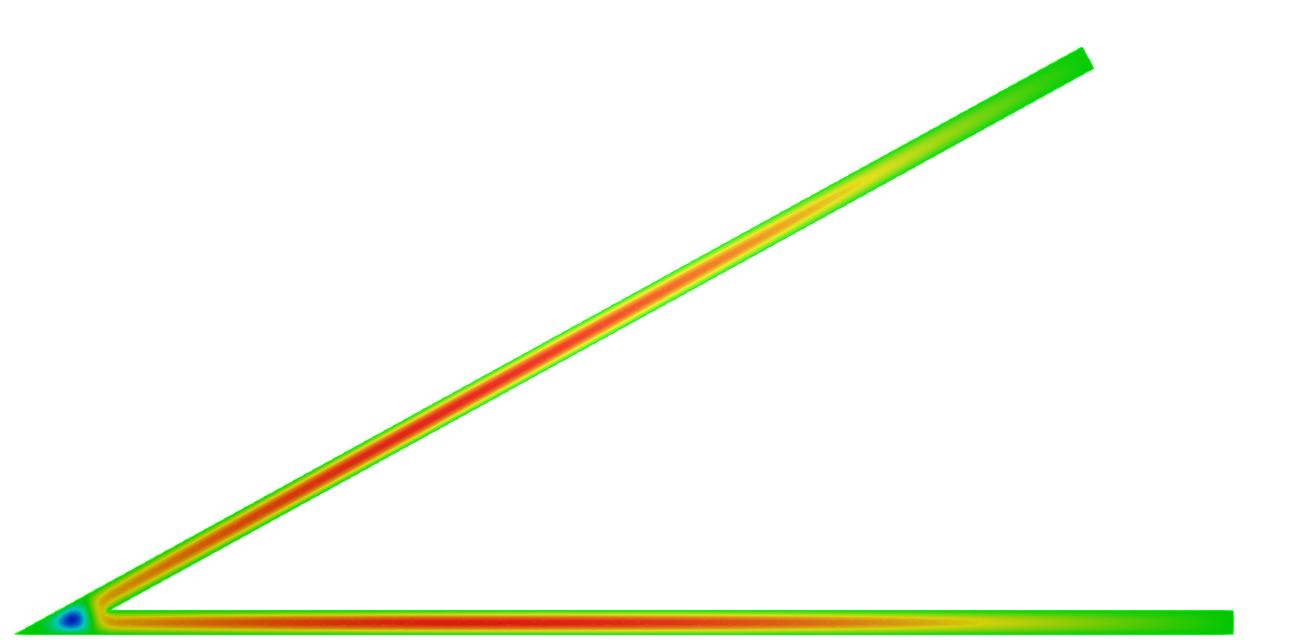}\  \includegraphics[width=5.2cm]{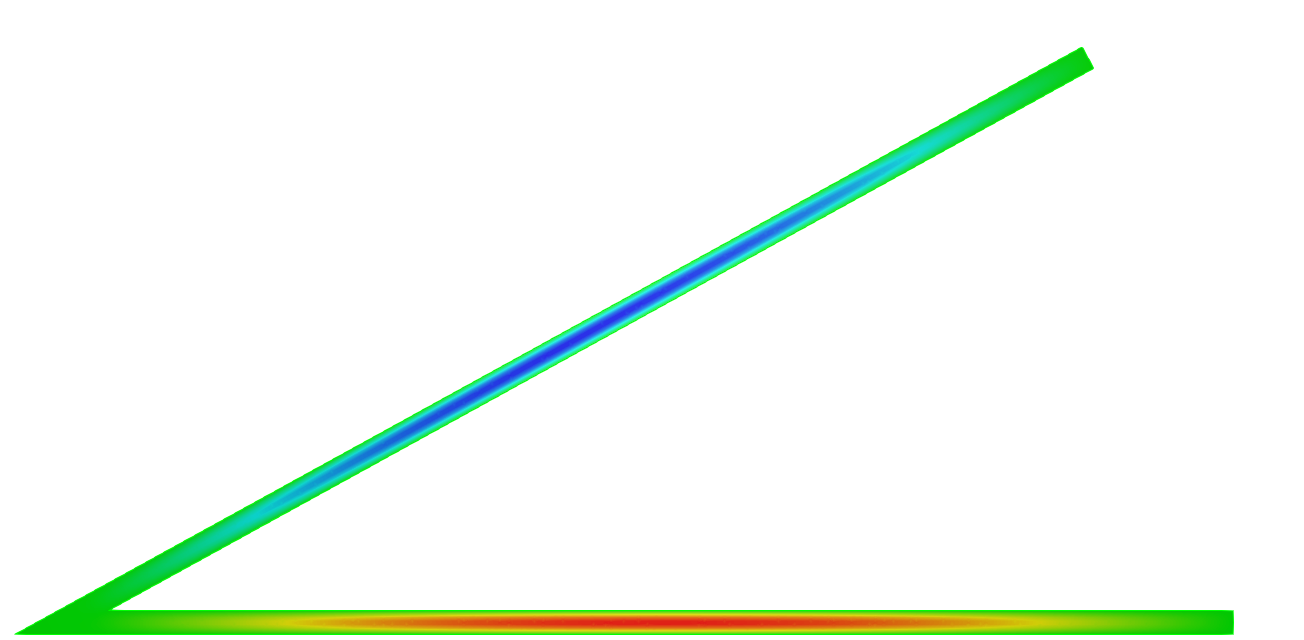}\  \includegraphics[width=5.2cm]{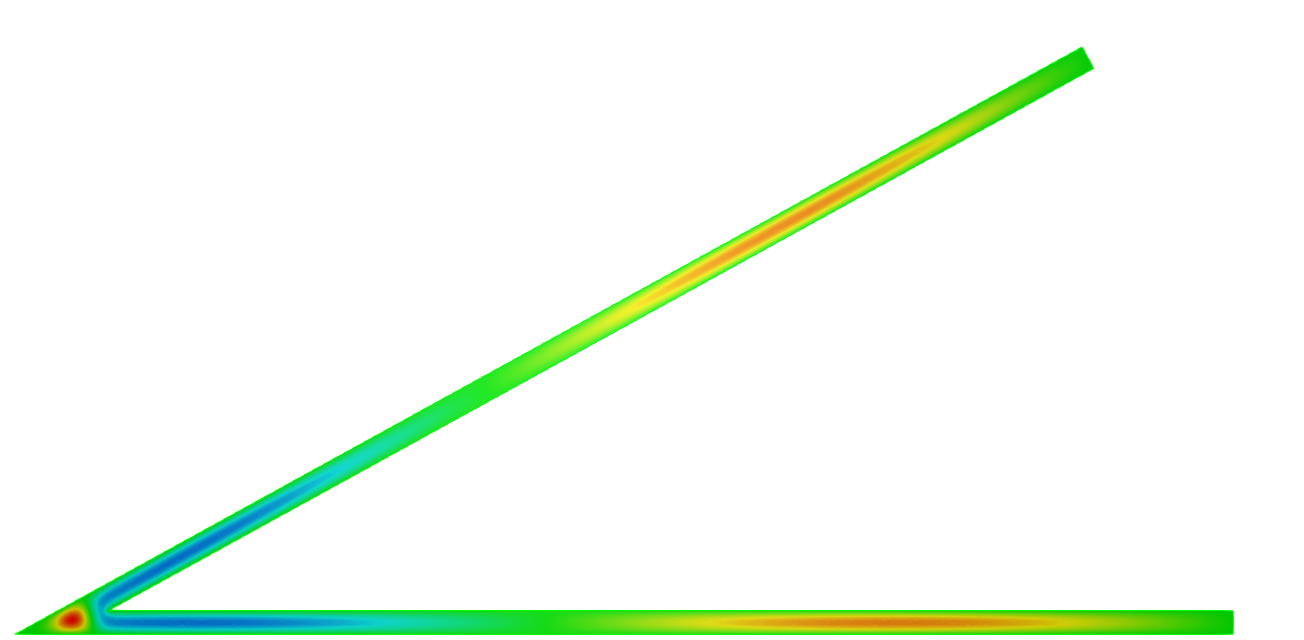}\\[3pt]
\includegraphics[width=5.2cm]{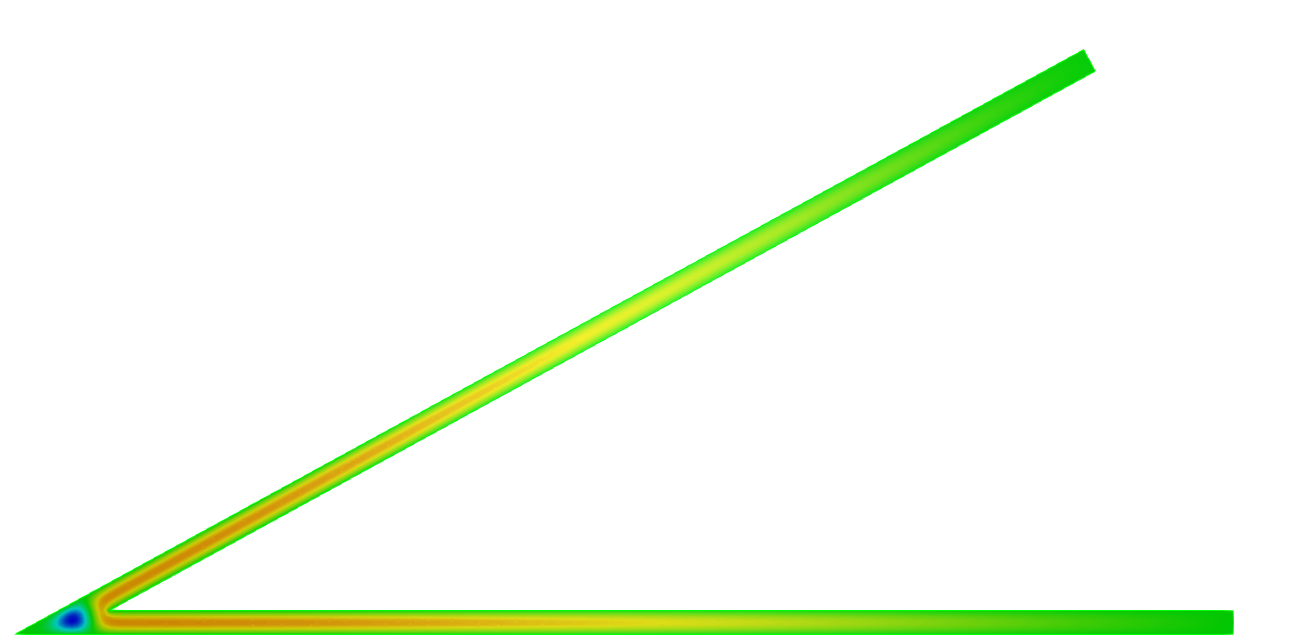}\  \includegraphics[width=5.2cm]{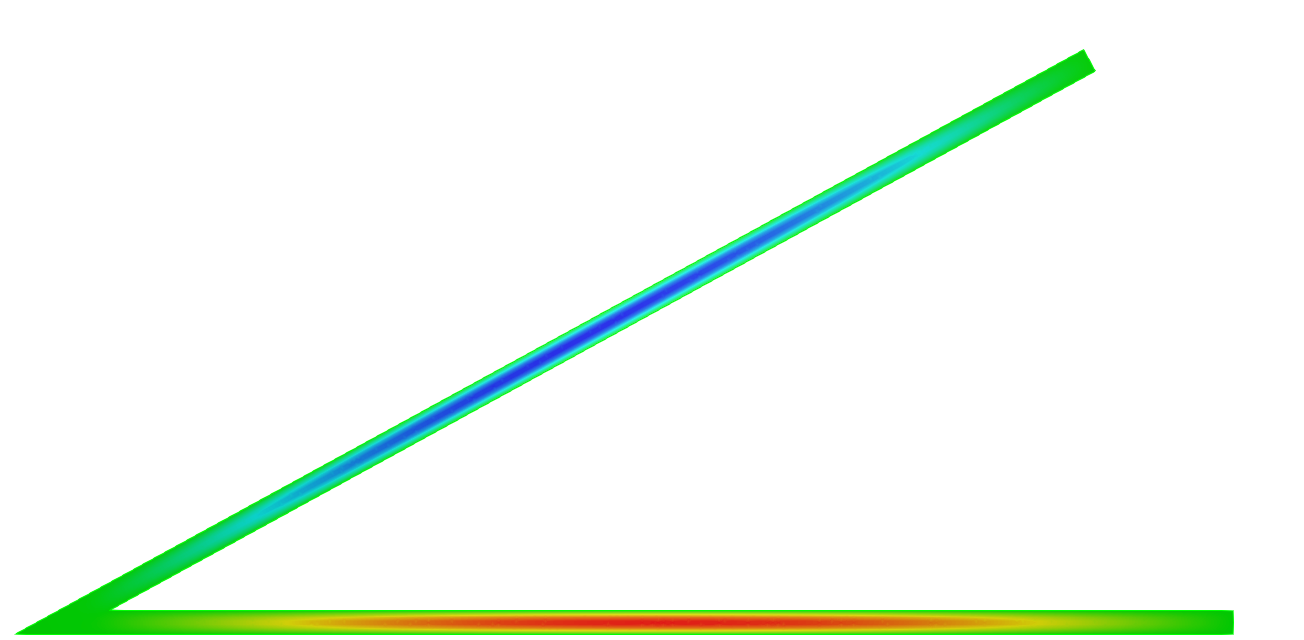}\  \includegraphics[width=5.2cm]{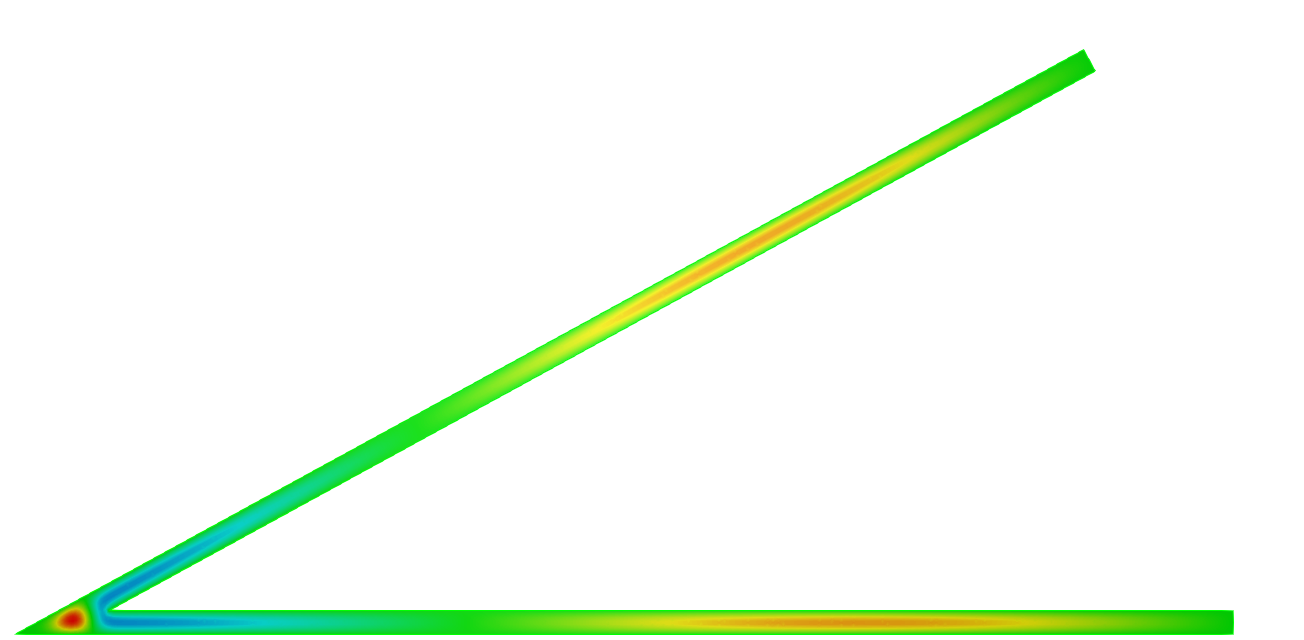}\\[3pt]
\includegraphics[width=5.2cm]{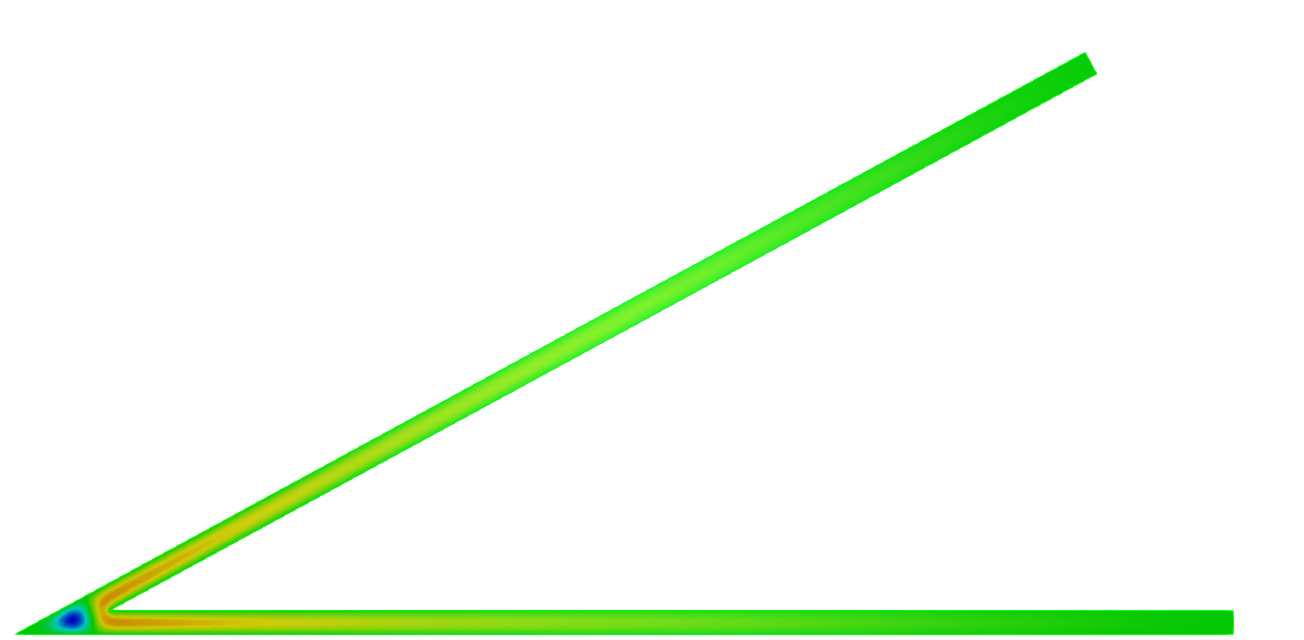}\  \includegraphics[width=5.2cm]{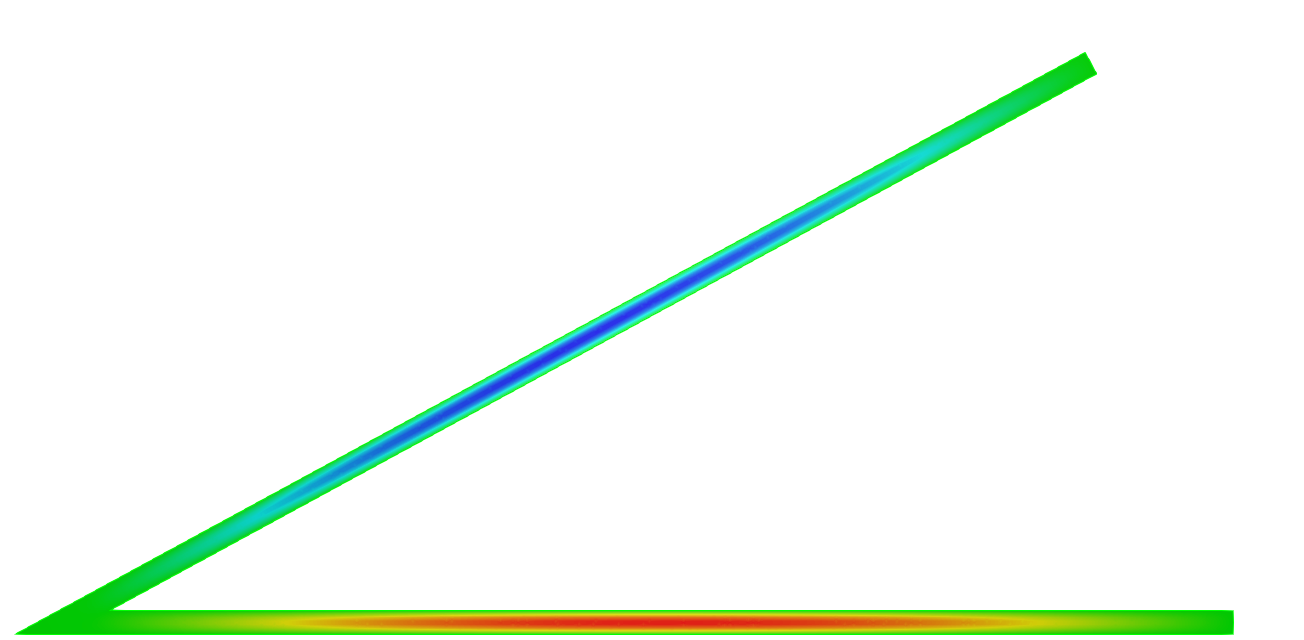}\  \includegraphics[width=5.2cm]{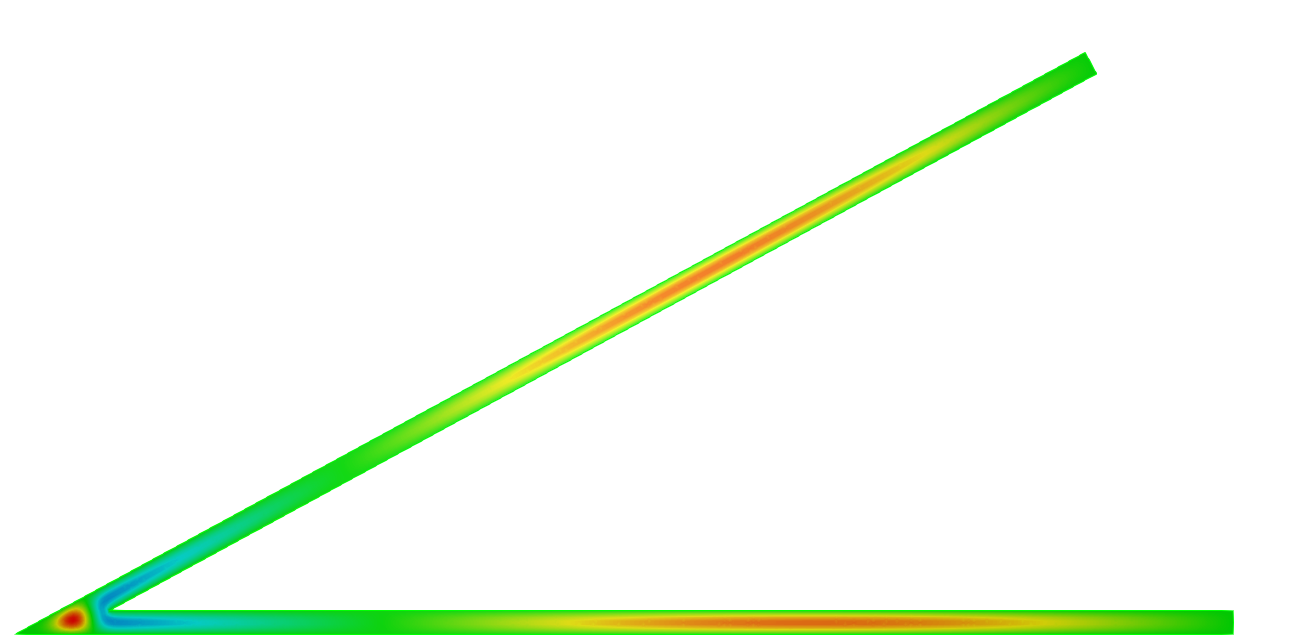}\\[3pt]
\includegraphics[width=5.2cm]{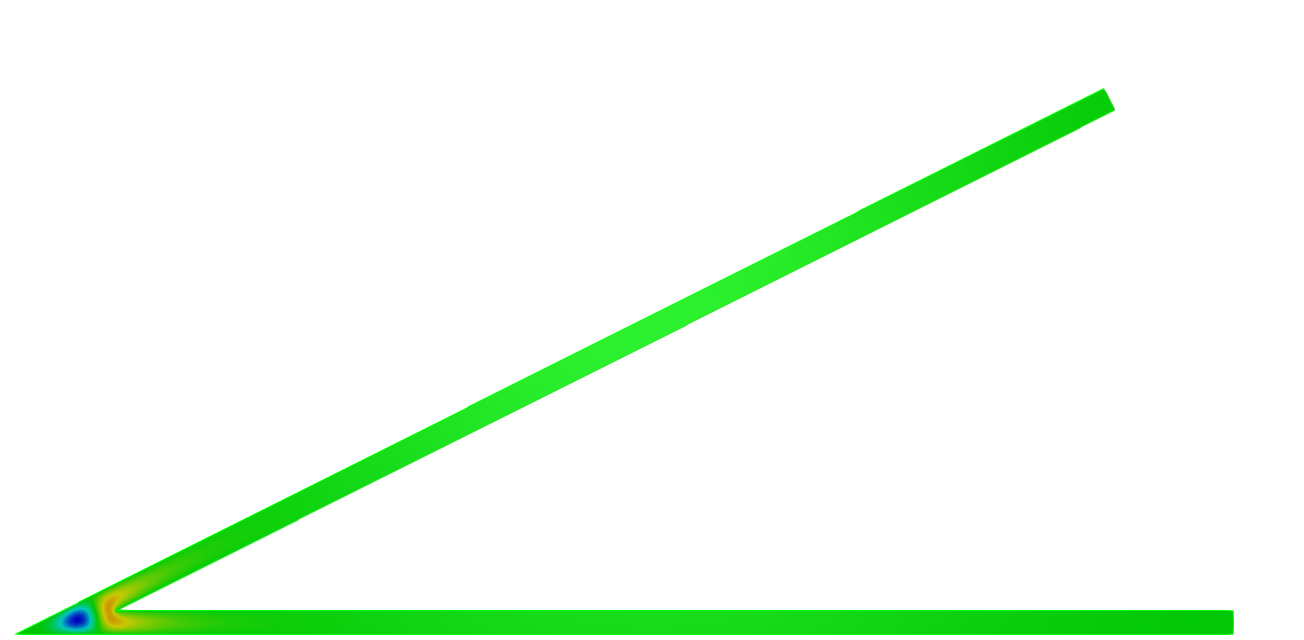}\  \includegraphics[width=5.2cm]{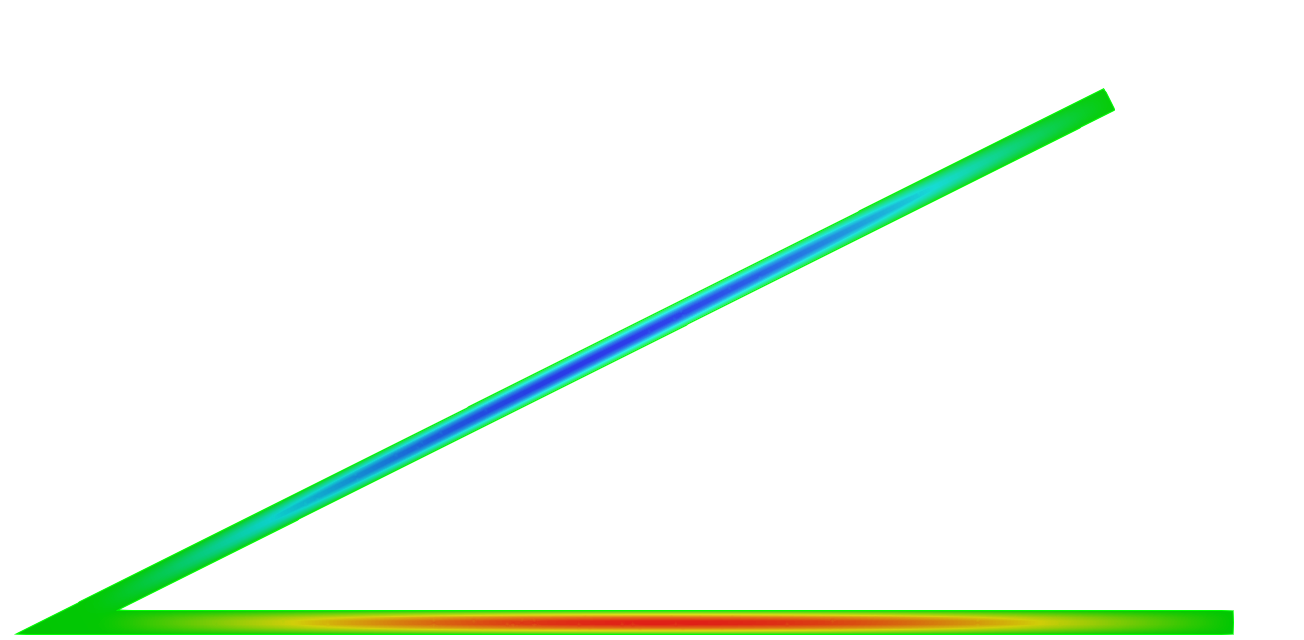}\  \includegraphics[width=5.2cm]{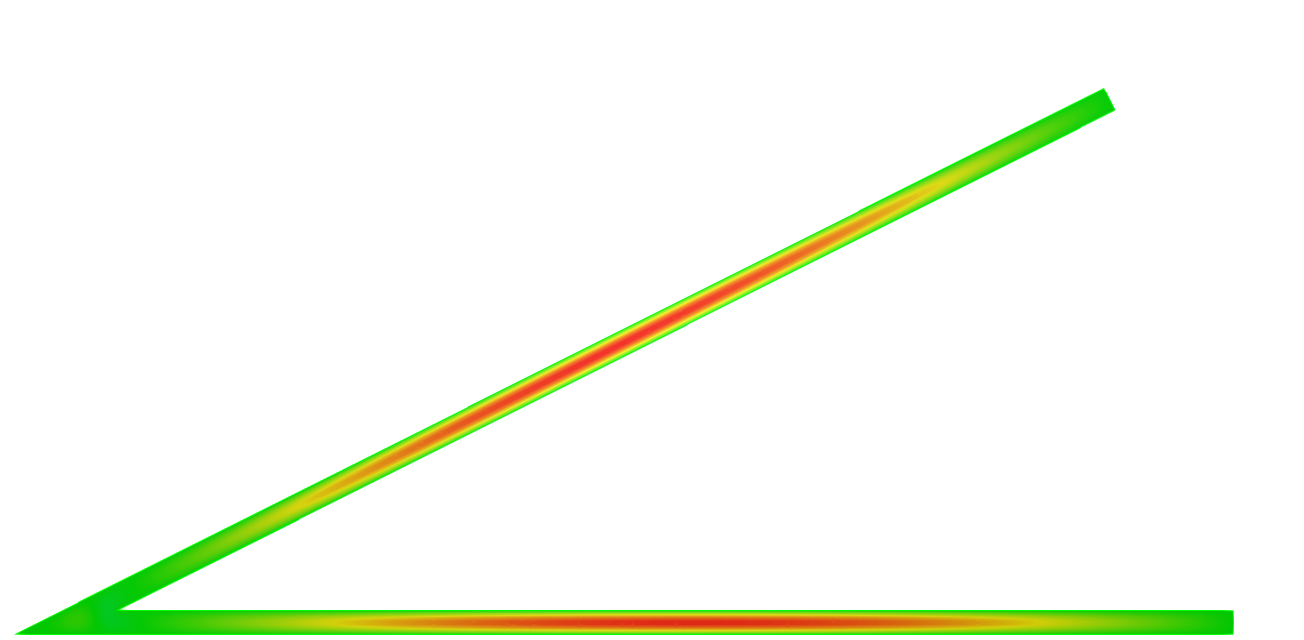}\\[3pt]
\caption{Eigenfunctions of the Dirichlet Laplacian in the broken strips obtained by symmetrization of the trapezoids $\mathrm{T}^\eps$ with respect to the line $\ell_{\alpha}=\{(x,y)\in\R^2\,|\,x=y\tan\alpha\}$. On each line, the geometry is the same and the eigenfunctions are associated with three different eigenvalues around $\pi^2/\eps^2$. From top to bottom, $\alpha$ is slightly increased around the first positive critical angle $\alpha^\star_1$.
\label{BrokenStrip}}
\end{figure}

\noindent Let us give a few numerical results in the initial broken strip represented in Figure \ref{GeomBrokenStrip} right.  More precisely, in Figure \ref{BrokenStrip} we display eigenfunctions of the Dirichlet Laplacian in domains obtained by symmetrization of the trapezoids $\mathrm{T}^\eps$ with respect to the line $\ell_{\alpha}\coloneqq\{(x,y)\in\R^2\,|\,x=y\tan\alpha\}$. From top to bottom, the angle $\alpha$ is slightly increased around the first positive critical $\alpha^\star_1$ and each column corresponds to a different eigenvalue. The behaviors of the eigenfunctions of the first and third columns, which are symmetric with respect to $\ell_{\alpha}$, are in agreement with the results presented in this work concerning Problem (\ref{MainProblem}) with mixed boundary conditions. In particular, the eigenfunction of the first column is the one corresponding to the eigenvalue which dives below $\pi^2/\eps^2$. While it spreads in the whole domain for $\alpha<\alpha^\star_1$, it becomes localized in the vicinity of the tip for $\alpha>\alpha^\star_1$. The eigenfunction of the third column has one node in the interior of each of the two branches when $\alpha<\alpha^\star_1$ whereas it has zero for $\alpha>\alpha^\star_1$. On the other hand, the eigenfunction of the second column, which is skew-symmetric with respect to $\ell_{\alpha}$, does not present particular changes around $\alpha^\star_1$.
\section{Two auxiliary results}\label{SectionAuxRes}
In this section, we first establish a technical result needed in the derivation of the model problem for $\alpha$ varying around a positive critical angle (see (\ref{CalculusB})).

\begin{proposition}\label{PropositionSigneIntegral}
For $\alpha\in(0;\pi/2)$, let $v$ be a non zero function of the space $\mX_{\mrm{bo}}$ defined in (\ref{DefSpaceX}). There holds
\begin{equation}\label{MainIntegral}
\int_{\Gamma}s\left(  |\partial_sv|^2-\pi^2|v|^2\right)\,ds>0.
\end{equation}
\end{proposition}
\begin{proof}
Let $v$ be an element of $\mX_{\mrm{bo}}\setminus\{0\}$. By working on the real and imaginary parts of $v$, without lost of generality, we can assume that $v$ is real valued. The function $v$ is bounded, belongs to $\mH^1_0(\Om_L;\Sigma_L)$ for all $L>0$ (see the definition of that space after (\ref{Pb_Approximation})) and satisfies (\ref{NearFieldPb}).\\
\newline
First let us formally apply the classical theory of regularity of solutions to elliptic problems in domains with singular geometries, see \textit{e.g.} \cite{Kond67,KoMR97}, to check that the integral (\ref{MainIntegral}) is well defined. 
Set $P_1\coloneqq(0,0)$, $P_2\coloneqq(\tan\alpha,1)$ (the two non smooth points of $\partial\Om$) and for $i=1,2$, introduce $(r_i,\theta_i)$ polar coordinates associated with $P_i$. At $P_1$, $P_2$, the principal singularities locally in $\mH^1$ write respectively
\begin{equation}\label{CalculSingularites}
r^{\lambda_1}\sin(\lambda_1\theta_1),\quad r^{\lambda_2}\sin(\lambda_2\theta_2)\qquad\mbox{with }\qquad\lambda_1\coloneqq\frac{\pi}{\pi-2\alpha},\quad \lambda_2\coloneqq\frac{\pi}{\pi+2\alpha}.
\end{equation}
A direct calculus guarantees that the function $r\mapsto r^{\nu}$ belongs to $\mL^2(0;1)$ if and only if $\nu>-1$. As a consequence, in order the integral (\ref{MainIntegral}) to be defined, we need to have both for $j=1$ and $2$
\[
1+2(\lambda_j-1)>-1\qquad\Leftrightarrow\qquad \lambda_j>0.
\]
From (\ref{CalculSingularites}), we see that this is true for all $\alpha\in(0;\pi/2)$.\\
Now we prove that the integral which appears in (\ref{MainIntegral}) is positive. To proceed, let us exploit the classical modification of the Rellich trick \cite{Rell43} due to C.S. Morawetz (see \textit{e.g.} \cite{MoLu68}). If $v\not\equiv0$ solves (\ref{NearFieldPb}), multiplying the EDP by $x\partial_xv$
and integrating by parts in $\Om$ (observe that $x\partial_xv$ decays exponentially at infinity), we get
\begin{equation}\label{Relation0}
\begin{array}{rcl}
0&=&\dsp\int_{\Om}\nabla v\cdot\nabla(x\partial_xv)-\pi^2xv\partial_xv\,dz\\[8pt]
&=&\dsp\int_{\Om}(\partial_xv)^2\,dz+\cfrac{1}{2}\int_{\Om}x\partial_x(|\nabla v|^2-\pi^2v^2)\,dz.
\end{array}
\end{equation}
Next, integrating by parts in $\Om_L=\{z\in\Om\,|\,x<L\}$ because \textit{a priori} $v$ does not decay at infinity, we can write
\[
\begin{array}{ll}
&\dsp\int_{\Om}x\partial_x(|\nabla v|^2-\pi^2v^2)\,dz  =  \underset{L\to+\infty}{\lim} \dsp\int_{\Om_L}x\partial_x(|\nabla v|^2-\pi^2v^2)\,dz \\[10pt]
 = & -\dsp\cos\alpha\int_{\Gamma}x(|\nabla v|^2-\pi^2v^2)\,ds-\underset{L\to+\infty}{\lim} \dsp\int_{\Om_L}|\nabla v|^2-\pi^2v^2\,dz-\int_{\Sigma_L}x(|\nabla v|^2-\pi^2v^2)\,dy\\[10pt]
  = & -\dsp\cos\alpha\int_{\Gamma}x(|\nabla v|^2-\pi^2v^2)\,ds- \dsp\int_{\Om}|\nabla v|^2-\pi^2v^2\,dz.
\end{array}
\]
Note that on $\Gamma$, since $\partial_nv=0$, we have $|\nabla v|=|\partial_sv|$. On the other hand, multiplying the identity $\Delta v+\pi^2v=0$ in $\Om_L$ by $v$, integrating by parts and taking the limit $L\to+\infty$, we obtain 
\begin{equation}\label{Relation3}
\int_{\Om}|\nabla v|^2-\pi^2v^2\,dz=0.
\end{equation}
By combining (\ref{Relation0})--(\ref{Relation3}),  we arrive at
\[
\cos\alpha\,\int_{\Gamma}x\,((\partial_s v)^2-\pi^2v^2)\,ds=2\dsp\int_{\Om}(\partial_xv)^2\,dz
\]
(again we emphasize that for $v\in\mX_{\mrm{bo}}$, the partial derivative of $v$ with respect to $x$ decays exponentially at infinity and so belongs to $\mL^2(\Om)$). Finally, since we have $x=s\sin\alpha$ on $\Gamma$, we find
\[
\int_{\Gamma}s\left((\partial_sv)^2-\pi^2v^2\right)\,ds=\cfrac{2}{\cos\alpha\sin\alpha}\,\dsp\int_{\Om}(\partial_xv)^2\,dz.
\]
Note that the right-hand side above is positive. Indeed otherwise we would have $\partial_xv=0$ on $\Gamma$ and so $\partial_sv=0$ on $\Gamma$. In that case, the conditions $v(P_1)=v(P_2)=0$ would imply $v=0$ on $\Gamma$. But we also have $\partial_nv=0$ on $\Gamma$. From the unique continuation principle (see \textit{e.g.} \cite{Bers}, \cite[\S8.3]{CoKr98} or \cite{Rond09}), this would give $v\equiv0$ in $\Om$.
\end{proof}
\noindent Finally we establish Proposition \ref{PropoDiscrete}.\\
\newline
\textit{Proof of Proposition \ref{PropoDiscrete}.}
Consider some positive critical angle $\alpha^\star$ and for simplicity denote $\Om$, $\Om^\eps$ the domain $\Om(\alpha^\star)$, $\Om(\alpha^\star+\eps)$. Let $W^\eps$ be the solution introduced in (\ref{ScaSol}) defined in $\Om^\eps$. It admits the decomposition 
\[
W^\eps  = w^{\mrm{in}}+\mathbb{S}^\eps w^{\mrm{out}}+\widetilde{W}^\eps,
\]
with $\mathbb{S}^\eps$ belonging to the unit circle in the complex plane and $\widetilde{W}^\eps\in\mH^1(\Om^\eps)$. Let us compute asymptotic expansions of  $W^\eps$, $\mathbb{S}^\eps$ as $\eps\to0$. It is natural to work  with the ans\"atze
\begin{equation}\label{ExpansionScaSol}
W^\eps=W+\eps W'+\dots,\qquad\qquad \mathbb{S}^\eps=\mathbb{S}+\eps \mathbb{S}'+\dots
\end{equation}
where $W$, $\mathbb{S}$ are the quantities appearing in (\ref{ScaSol}) and $W'$, $\mathbb{S}'$ are to be determined. Inserting (\ref{ExpansionScaSol}) into the problem satisfied by $W^\eps$, taking the limit $\eps\to0$ and collecting the terms of order $\eps$, by exploiting condition (\ref{BdryConditionP}), we find that $W'$ must satisfy
\[
\begin{array}{|rcll}
\Delta W'+\pi^2W'&=&0&\mbox{ in }\Om\\
W'&=&0&\mbox{ on } \Sigma\\[3pt]
\partial_n W'  &=& -\partial_sW- s(\partial^2_sW+\pi^2W) & \mbox{ on }\Gamma.
\end{array}
\]
Next, we write
\[
\begin{array}{rcl}
0&=&\dsp\underset{L\to+\infty}{\lim} \dsp\int_{\Om_L}(\Delta W'+\pi^2W')\overline{W}-W'(\Delta \overline{W}+\pi^2\overline{W})\,dz\\[10pt]
& =& \dsp\int_{\Gamma}s\left(|\partial_sW|^2-\pi^2|W|^2\right)\,ds+\dsp\underset{L\to+\infty}{\lim}\int_{\Sigma_L}\partial_xW' \overline{W}-W'\partial_x\overline{W}\,dy\\[10pt]
& =& \dsp\int_{\Gamma}s\left(|\partial_sW|^2-\pi^2|W|^2\right)\,ds+i\mathbb{S}'\overline{\mathbb{S}}.
\end{array}
\]
Thus for the derivative of $\mathbb{S}$ with respect to $\alpha$, we obtain
\begin{equation}\label{ComputationDerivative}
\mathbb{S}'=i\mathbb{S}\dsp\int_{\Gamma}s\left(|\partial_sW|^2-\pi^2|W|^2\right)\,ds.
\end{equation}
Now let us exploit that $\alpha^\star$ is a positive critical angle. In that case we know from (\ref{ProofBounded}) that $W$ is bounded, \textit{i.e.} that $W$ belongs to $\mX_{\mrm{bo}}$. Then Proposition \ref{PropositionSigneIntegral} ensures that the integral appearing on the right-hand side of (\ref{ComputationDerivative}) is positive. This proves that
\[
\mathbb{S}'=\cfrac{\partial\mathbb{S}}{\partial\alpha}(\alpha^\star)\ne0.
\]
Formula (\ref{ComputationDerivative}) also shows that $\alpha\mapsto \mathbb{S}(\alpha)$ runs counter-clockwise on the unit circle in the complex plane when $\alpha$ increases in a neighborhood of $\alpha^\star$, which is indeed what we observe in the numerics of Figure \ref{TRPhase}. On the other hand, this guarantees that there holds $\mathbb{S}(\alpha)\ne-1$ for $\alpha\in[\alpha^\star-c;\alpha^\star+c]\setminus\{\alpha^\star\}
$ for $c>0$ small enough, which implies that the set of critical angles is discrete. \hfill $\square$

\section{Concluding remarks}

\noindent In this work, we considered the spectrum of the  Laplace operator with mixed boundary conditions in thin trapezoids characterized by an angle $\alpha$. By working with symmetries, this allows one to study the spectrum of the Dirichlet Laplacian in thin broken strips as pictured in Figure \ref{BrokenStrip}. When varying  continuously $\alpha$, at certain particular values $\alpha^\star_k$, $k\in\N$, we proved a phenomenon of eigenvalue fall around the normalized threshold $\pi^2/\eps^2$. The fall is milder at $\alpha^\star_0=0$ than at $\alpha^\star_k>0$, $k\ge1$. Moreover, the eigenfunctions associated with the eigenvalue diving below $\pi^2/\eps^2$ experiment a notable change of shape: while they spread in the whole domain $\mathrm{T}^\eps$ for $\alpha<\alpha^\star_k$, they become localized in the vicinity of the tip for $\alpha>\alpha^\star_k$. Let us mention that this is very particular to Dirichlet BC and does not appear with Neumann BC. Indeed, for Neumann BC the continuous spectrum of the near field operator starts at zero and the corresponding $\mX_{\dagger}$ is always of dimension one (due to constants), independently of the geometry.\\
\newline
Around this study, several questions remain open. First, we assumed that there are positive critical angles $\alpha^\star_k$, $k\ge1$, where $\mX_{\dagger}\ne0$. Though we observe them numerically (see again Figure \ref{TRPhase}), we do not have a proof that there exists at least one. When $\alpha$ increases, $\alpha\mapsto \mathbb{S}(\alpha)$ runs continuously on the unit circle in the complex plane but Figure \ref{TRPhase} indicates that it is not monotone, which complicates the analysis. One idea maybe would be to look for an asymptotic expansion of $\mathbb{S}(\alpha)$ as $\alpha$ tends to zero to obtain an explicit behavior allowing one to conclude to a passage through $-1$. Another intriguing question concerns the space $\mX_{\mrm{tr}}$ of trapped modes at the threshold for the operator $A^\Om$. Can we have $\mX_{\mrm{tr}}\ne\{0\}$ for certain angles? This is unclear. To pursue this analysis, it might be interesting to investigate the spectrum of the Dirichlet Laplacian in 3D polyhedral domains of thickness $\eps$ small. Can one write models on the limit polyhedron to describe the asymptotics of the eigenvalues as $\eps$ tends to zero? Are they defined on the faces or on the edges? How do they depend on the different angles of the polyhedron?

\section*{Acknowledgements}

The work of the second author was supported by the Ministry of Science and Higher Education of Russian Federation (Project 124041500009-8 and agreement 075-15-2025-344 dated 29/04/2025 for Saint Petersburg Leonhard Euler International Mathematical Institute at PDMI RAS).

\bibliography{Bibli}

@preamble{
   "\def\cprime{$'$} "
}

@article{ViLu,
  title={Regular degeneration and boundary layer for linear differential equations with small parameter},
  author={Vishik, M.I. and Lyusternik, L.A.},
  journal={Uspekhi Mat. Nauk},
  volume={12},
  number={5},
  pages={3--122},
  year={1957},
}

@article{BoCF25,
 title		 = "Trapped modes in electromagnetic waveguides",
 author		 = "{Bonnet-Ben Dhia}, {A.-S.} and Chesnel, L. and Fliss, S.",
 journal	 = "arXiv preprint arXiv:2512.17763",
 year		 = "2025"}

@article{Frei07,
  title={{Precise bounds and asymptotics for the first Dirichlet eigenvalue of triangles and rhombi}},
  author={Freitas, P.},
  journal={J. Funct. Anal.},
  volume={251},
  number={1},
  pages={376--398},
  year={2007},
}

@article{FrSo09,
  title={{On the spectrum of the Dirichlet Laplacian in a narrow strip}},
  author={Friedlander, L. and Solomyak, M.},
  journal={Isr. J. Math.},
  volume={170},
  pages={337--354},
  year={2009},
}

@article{BoFr09,
 author = {Borisov, D. and Freitas, P.},
 title = {Singular asymptotic expansions for {Dirichlet} eigenvalues and eigenfunctions of the {Laplacian} on thin planar domains},
 journal = {Ann. Inst. Henri Poincar{\'e}, Anal. Non Lin{\'e}aire},
 volume = {26},
 number = {2},
 pages = {547--560},
 year = {2009},
}

@article{DuEx95,
author = {Duclos, P. and Exner, P.},
title = {Curvature-induced bound states in quantum waveguides in two and three dimensions},
journal = {Rev. Math. Phys.},
volume = {07},
number = {01},
pages = {73-102},
year = {1995}
}

@article{ExSe89,
  title={Bound states in curved quantum waveguides},
  author={Exner, P. and Seba, P.},
  journal={J. Math. Phys.},
  volume={30},
  number={11},
  pages={2574--2580},
  year={1989}  
}

@article{GPRO10,
 title		 = "Bound states in the continuum in graphene quantum dot structures",
 author		 = "Gonz{\'a}lez, {J.W.} and Pacheco, M. and Rosales, L. and Orellana, {P.A.}",
 journal	 = "Europhys. Lett.",
 volume		 = "91",
 number		 = "6",
 pages		 = "66001",
 year		 = "2010",
}

@article{Mois09,
 title		 = "{Suppression of Feshbach resonance widths in two-dimensional waveguides and quantum dots: a lower bound for the number of bound states in the continuum}",
 author		 = "Moiseyev, N.",
 journal	 = "Phys. Rev. Lett.",
 volume		 = "102",
 number		 = "16",
 pages		 = "167404",
 year		 = "2009",
}

@article{SaBR06,
 title		 = "Bound states in the continuum in open quantum billiards with a variable shape",
 author		 = "Sadreev, {A.F.} and Bulgakov, {E.N.} and Rotter, I.",
 journal	 = "Phys. Rev. B",
 volume		 = "73",
 number		 = "23",
 pages		 = "235342",
 year		 = "2006",
}

@article{GoJa92,
  title={Bound states in twisting tubes},
  author={Goldstone, J. and Jaffe, R.L.},
  journal={Phys. Rev. B},
  volume={45},
  number={24},
  pages={14100},
  year={1992},
  publisher={APS}
}

@article{CLMM93,
  title={Multiple bound states in sharply bent waveguides},
  author={Carini, {J.P.} and Londergan, {J.T.} and Mullen, K. and Murdock, {D.P.}},
  journal={Phys. Rev. B},
  volume={48},
  number={7},
  pages={4503},
  year={1993}  
}

@article{NGPNG09,
  title={The electronic properties of graphene},
  author={Neto, C. and Guinea, F. and Peres, N. and Novoselov, {K.S.} and Geim, {A.K.}},
  journal={Rev. Mod. Phys.},
  volume={81},
  number={1},
  pages={109},
  year={2009},  
}

@article{GeNo07,
  title={The rise of graphene},
  author={Geim, {A.K.} and Novoselov, {K.S.}},
  journal={Nature materials},
  volume={6},
  number={3},
  pages={183--191},
  year={2007},  
}

@article{Geim09,
  title={Graphene: status and prospects},
  author={Geim, {A.K.}},
  journal={Science},
  volume={324},
  number={5934},
  pages={1530--1534},
  year={2009},  
}

@article{MoLu68,
  title={An inequality for the reduced wave operator and the justification of geometrical optics},
  author={Morawetz, {C.S.} and Ludwig, D.},
  journal={Commun. Pure Appl. Math.},
  volume={21},
  number={2},
  pages={187--203},
  year={1968},  
}

@book{Bers,
 AUTHOR 	 = "Bers, L. and John, F. and Schechter, M.",
 TITLE		 = "Partial differential equations", 
 PUBLISHER   = "AMS",
 YEAR 		 = "1964", 
}

@article{Rond09,
  title={{The stability for the Cauchy problem for elliptic equations}},
  author={Alessandrini, G. and Rondi, L. and Rosset, E. and Vessella, S.},
  journal={Inverse problems},
  volume={25},
  number={12},
  pages={123004},
  year={2009}  
}

@article{NaSh14,
  title={{Trapped modes in angular joints of 2D waveguides}},
  author={Nazarov, {S.A.} and Shanin, {A.V}},
  journal={Appl. Anal.},
  volume={93},
  number={3},
  pages={572--582},
  year={2014},
}

@book{londergan1999binding,
  title={Binding and scattering in two-dimensional systems: applications to quantum wires, waveguides and photonic crystals},
  author={Londergan, {J.T.} and Carini, {J.P.} and Murdock, {D.P.}},
  volume={60},
  year={1999},
  publisher={Springer Science \& Business Media}
}

@article{ChNa18,
 title		 = "Non reflection and perfect reflection via {Fano} resonance in waveguides",
 author		 = "Chesnel, L. and Nazarov, {S.A.}",
 journal	 = "Comm. Math. Sci.",
 volume		 = "16",
 number		 = "7",
 pages		 = "1779--1800",
 year		 = "2018",
}

@article{ShTu12,
 title		 = "Total resonant transmission and reflection by periodic structures",
 author		 = "Shipman, {S.P.} and Tu, H.",
 journal	 = "SIAM J. Appl. Math.",
 volume		 = "72",
 number		 = "1",
 pages		 = "216--239",
 year		 = "2012",
}

@article{Post05,
  title={{Branched quantum wave guides with Dirichlet boundary conditions: the decoupling case}},
  author={Post, O.},
  journal={J. Phys. A Math. Theor.},
  volume={38},
  number={22},
  pages={4917},
  year={2005},  
}

@book{ExKo15,
  title={Quantum waveguides},
  author={Exner, P. and Kova{\v{r}}{\'\i}k, H.},
  year={2015},
  publisher={Springer}
}

@article{NaRU15,
  title={{Asymptotics of the spectrum of the Dirichlet Laplacian on a thin carbon nano-structure}},
  author={Nazarov, {S.A.} and Ruotsalainen, K. and Uusitalo, P.},
  journal={C. R. M\'ecanique},
  volume={343},
  number={5-6},
  pages={360--364},
  year={2015},
}

@BOOK{MaNaPl,
 title	 	 = "{Asymptotic theory of elliptic boundary value problems in singularly perturbed domains, Vol. 1}",
 author		 = "{Maz'ya}, {V.G.} and Nazarov, {S.A.} and Plamenevski{\u\i}, {B.A.}",
 year 	 	 = "2000",
 publisher	 = "{Birkh\"{a}user}, Basel"
}

@book {Ilin,
    AUTHOR = {Il'in, A. M.},
     TITLE = {Matching of asymptotic expansions of solutions of boundary
              value problems},
    SERIES = {Translations of Mathematical Monographs},
    VOLUME = {102},
 PUBLISHER = {AMS},
   ADDRESS = {Providence, RI},
      YEAR = {1992},
     PAGES = {x+281},
      ISBN = {},
   MRCLASS = {},
  MRNUMBER = {},
}

@book{VD,
 AUTHOR		 = "Van Dyke, M.",
 TITLE		 = "Perturbation methods in fluid mechanics",
 EDITION	 = "",
 PUBLISHER	 = "The Parabolic Press, Stanford, Calif.",
 YEAR		 = "1975",
 PAGES 		 = "xiv+271",
}

@article{ChNa23,
  title={{Spectrum of the Dirichlet Laplacian in a thin cubic lattice}},
  author={Chesnel, L. and Nazarov, {S.A.}},
  journal	 = "Math. Model. Numer. Anal.",
  volume      = "57",  
 pages		 = "3251--3273",
 year		 = "2023"
}

@article{ChNa25,
  title={On the breathing of spectral bands in periodic quantum waveguides with inflating resonators},
  author={Chesnel, L. and Nazarov, {S.A.}},
  journal	 = "Math. Model. Numer. Anal.",
  volume     = "59", 
  number	 = "4",
 pages		 = "2171--2206",
 year		 = "2025"
}

@article{Naza18,
  title={Breakdown of cycles and the possibility of opening spectral gaps in a square lattice of thin acoustic waveguides},
  author={Nazarov, {S.A.}},
  journal={Izv. Math.},
  volume={82},
  number={6},
  pages={1148--1195},
  year={2018},  
  NOTE = "transl. from Izv. Ross. Akad. Nauk, Ser. Mat. 82,6:78--127, 2018",  
}

@article{Naza23,
  title={{On the one-dimensional asymptotic models of thin Neumann lattices}},
  author={Nazarov, {S.A}},
  journal={Sib. Math. J.},
  volume={64},
  number={2},
  pages={356--373},
  year={2023},  
}

@InCollection{MoVa08,
 Author = {Molchanov, S. and Vainberg, B.},
 Title = {Laplace operator in networks of thin fibers: spectrum near the threshold},
 BookTitle = {Stochastic analysing in mathematical physics. Proceedings of a satellite conference of ICM 2006, Lisbon, Portugal, September 4--8, 2006. Selected papers.}, 
 Pages = {69--94},
 Year = {2008},
}

@article {Naza17bis,
 AUTHOR		 = "Nazarov, {S.A.}",
 TITLE		 = "Almost standing waves in a periodic waveguide with resonator, and near-threshold eigenvalues",
 JOURNAL	 = "Algebra i Analiz",
 VOLUME		 = "28",
 YEAR		 = "2016",
 NUMBER		 = "3",
 PAGES		 = "110--160",
 note		 = "(English transl.: Sb. Math. J. 2017. V. 28, N 3. P. 377--410)" 
}

@article {Hech12,
    AUTHOR = "Hecht, F.", 
	TITLE = "New development in FreeFem++",
   JOURNAL = "J. Numer. Math.",     
    VOLUME = "20", 
	YEAR = "2012",
    NUMBER = "3-4", 
	PAGES = "251--265",      
	note = "\url{https://freefem.org/}"
}

@article{Naza12a,
 AUTHOR		 = "Nazarov, {S.A.}",
 TITLE		 = "Discrete spectrum of cranked, branching, and periodic waveguides",
 JOURNAL	 = "Algebra i analiz",
 VOLUME		 = "23",
 number		 = "2",
 YEAR		 = "2011", 
 pages       = "206--247",
 note		 = "(English transl.: St. Petersburg Math. 23:2, 351--379, 2012.)" 
}

@article{Naza14c,
  title={Asymptotics of eigenvalues of the {Dirichlet} problem in a skewed $\mathscr{T}$-shaped waveguide},
  author={Nazarov, {S.A.}},
  journal={Comput. Math. Math. Phys.},
  volume={54},
  pages={811--830},
  year={2014},
}

@article{Rell43,
 Author = {Rellich, F.},
 Title = {{\"U}ber das asymptotische {Verhalten} der {L{\"o}sungen} von {{\(\Delta u+\lambda u=0\)}} in unendlichen {Gebieten}},
 Journal = {Jahresber. Dtsch. Math.-Ver.}, 
 Volume = {53},
 Pages = {57--65},
 Year = {1943}
}

@article{ExL,
 AUTHOR		 = "Exner, P. and Seba, P. and Stovicek, P.",
 TITLE		 = "On existence of a bound state in an {L}-shaped waveguide",
 JOURNAL	 = "Czech J. Phys.",
 VOLUME		 = "39",
 YEAR		 = "1989",  
 pages       = "1181--1191"
}

@article {KaNa00,
 AUTHOR		 = "Kamotskii, {I.V.} and Nazarov, {S.A.}",
 TITLE		 = "On eigenfunctions localized in a neighborhood of the lateral surface of a thin domain",
 JOURNAL	 = "J. Math. Sci.",
 VOLUME		 = "101",
 YEAR		 = "2000",
 NUMBER		 = "2",
 PAGES		 = "2941--2974"
}

@article{DaRa12,
  title={Plane waveguides with corners in the small angle limit},
  author={Dauge, M. and Raymond, N.},
  journal={J. Math. Phys.},
  volume={53},
  number={12},
  pages={123529},
  year={2012},  
}

@book{RS78,
  title={Methods of modern mathematical physics. IV. Analysis of operators},
  author={Reed, M. and Simon, B.},
  year={1978},
  publisher={Academic Press, New-York}
}

@BOOK{Lad,
  AUTHOR      = "Ladyzhenskaya, {O.A.}",
  TITLE       = "The Boundary Value Problems of Mathematical Physics",
  YEAR        = "1973",
  SERIES      = "",
  VOLUME      = "", 
  PUBLISHER   = "Nauka",
  ADDRESS     = "Moscow",
  note		 = "(English transl.: Springer-Verlag, New York, 1985)" 
}

@article{Naza16,
  title={Transmission conditions in one-dimensional model of a rectangular lattice of thin quantum waveguides},
  author={Nazarov, {S.A.}},
  journal={J. Math. Sci.},
  volume={219},
  number={6},
  pages={994--1015},
  year={2016},  
}

@article{ABGM91,
  title={Quantum bound states in open geometries},
  author={Avishai, Y. and Bessis, D. and Giraud, {B.G.} and Mantica, G. },
  journal={Phys. Rev. B},
  volume={15},
  number={15},
  pages={8028--8034},
  year={1991},  
}

@article{Naza14,
  title={Bounded solutions in a {T}-shaped waveguide and the spectral properties of the Dirichlet ladder},
  author={Nazarov, {S.A.}},
  journal={Comput. Math. Math. Phys.},
  volume={54},
  number={8},
  pages={1261--1279},
  year={2014},  
}

@article{Grie08,
  title={Spectra of graph neighborhoods and scattering},
  author={Grieser, D.},
  journal={Proc. Lond. Math. Soc.},
  volume={97},
  number={3},
  pages={718--752},
  year={2008},  
}

@article{MoVa07,
  title={Scattering solutions in networks of thin fibers: small diameter asymptotics},
  author={Molchanov, S. and Vainberg, B.},
  journal={Commun. Math. Phys.},
  volume={273},
  number={2},
  pages={533--559},
  year={2007},  
}

@book{Post12,
  title={Spectral analysis on graph-like spaces},
  author={Post, O.},
  volume={2039},
  year={2012},  
  publisher={Springer Science \& Business Media}
}

@book{BeKu13,
  title={Introduction to quantum graphs},
  author={Berkolaiko, G. and Kuchment, P.},
  volume={186},
  year={2013},  
  Publisher = {Providence, RI: American Mathematical Society (AMS)},
}

@article{Kuch02,
  title={Graph models for waves in thin structures},
  author={Kuchment, P.},
  journal={Waves in random media},
  volume={12},
  number={4},
  pages={R1},
  year={2002},  
}

@article{Naza17,
  title={The spectra of rectangular lattices of quantum waveguides},
  author={Nazarov, {S.A.}},
  journal={Izv. Math.},
  volume={81},
  number={1},
  pages={29},
  year={2017},
}

@article{Pank17,
  title={Eigenvalue inequalities and absence of threshold resonances for waveguide junctions},
  author={Pankrashkin, K.},
  journal={J. Math. Anal. Appl.},
  volume={449},
  number={1},
  pages={907--925},
  year={2017},  
}

@book {BiSo87,
 AUTHOR		 = "Birman, {M.Sh.} and Solomjak, {M.Z.}",
 TITLE		 = "Spectral theory of selfadjoint operators in {H}ilbert space",
 SERIES		 = "Mathematics and its Applications (Soviet Series)",
 PUBLISHER	 = "D. Reidel Publishing Co.",
 ADDRESS	 = "Dordrecht",
 YEAR		 = "1987",
 PAGES		 = "xv+301",
}

@article{Naza20Threshold,
  title={Threshold resonances and virtual levels in the spectrum of cylindrical and periodic waveguides},
  author={Nazarov, {S.A.}},
  journal={Izv. Math.},
  volume={84},
  number={6},
  pages={1105},
  year={2020},  
}

@article{KoNS16,
 title		 = "Stabilizing solutions at thresholds of the continuous spectrum and anomalous transmission of waves",
 author		 = "Korolkov, {A.I.} and Nazarov, {S.A.} and Shanin, {A.V.}",
 journal	 = "Z. Angew. Math. Mech.",
 volume		 = "96",
 number		 = "10",
 pages		 = "1245--1260",
 year		 = "2016"
}

@book {CoKr98,
 AUTHOR 	 = "Colton, D. and Kress, R.",
 TITLE		 = "Inverse acoustic and electromagnetic scattering theory",
 SERIES		 = "Applied Mathematical Sciences",
 VOLUME		 = "93",
 EDITION	 = "Second",
 PUBLISHER  = "Springer-Verlag",
 ADDRESS	 = "Berlin",
 YEAR 		 = "1998",
 PAGES		 = "xii+334"
}

@book{Henr06,
 title		 = "Extremum problems for eigenvalues of elliptic operators",
 author		 = "Henrot, A.",
 year		 = "2006",
 publisher	 = "{Birkh\"{a}user}"}

@ARTICLE{Kond67,
 AUTHOR  = "{Kondratiev}, {V.A.}",
 TITLE   = "Boundary-value problems for elliptic equations in domains with conical or angular points",
 JOURNAL = "Trans. Moscow Math. Soc.",
 YEAR    = "1967",
 VOLUME  = "16",
 PAGES   = "227-313"}

@Book{Kato95,
    Author = {Kato, T.},
    Title = {{Perturbation theory for linear operators.}},
    Edition = {Reprint of the corr. print. of the 2nd ed. 1980},
    Pages = {xxi + 619},
    Year = {1995},
 	publisher= {Springer-Verlag},
 address = {Berlin}
}

@BOOK{KoMR97,
  AUTHOR 	  = {Kozlov, {V.A.} and Maz'ya, {V.G.} and Rossmann, J.},
  TITLE       = {Elliptic Boundary Value Problems in Domains with Point Singularities},
  YEAR        = {1997},
  SERIES      = {Mathematical Surveys and Monographs},
  VOLUME      = {52}, 
  PUBLISHER   = {AMS},
  ADDRESS     = {Providence}}

@book{LiMa68,
 author 	 = "J.-L. Lions and E. Magenes",
 title 		 = "{Probl\`{e}mes} aux limites non {homog\`{e}nes} et applications",
 year  		 = "1968",
 publisher	 = "Dunod"}

@BOOK{NaPl94,
  AUTHOR      = "Nazarov, {S.A.} and Plamenevski{\u\i}, {B.A.}",
  TITLE       = "Elliptic problems in domains with piecewise smooth boundaries",
  YEAR        = "1994",
  SERIES      = "Expositions in Mathematics",
  VOLUME      = "13", 
  PUBLISHER   = "De Gruyter",
  ADDRESS     = "Berlin, Germany"
}
\bibliographystyle{plain}
\end{document}